\documentclass[12pt]{article}
\usepackage{amssymb}
\usepackage{amsmath}
\usepackage{latexsym}
\usepackage{mathrsfs}
\usepackage{graphicx}
\usepackage{caption}
\usepackage{comment}
\usepackage{amsthm}
\usepackage[titletoc,title]{appendix}
\usepackage{hyperref}
\usepackage{color}
\allowdisplaybreaks[1]

\numberwithin{equation}{section}
\newtheorem{theorem}[equation]{Theorem}
\newtheorem{proposition}[equation]{Proposition}

\newtheorem{lemma}[equation]{Lemma}

\providecommand{\norm}[1]{\left\lVert#1\right\rVert}

\title{On vibrating thin membranes with mass concentrated near the boundary: an asymptotic analysis}

\author{Matteo Dalla Riva\thanks{Department of Mathematics, The University of Tulsa, Tulsa, Oklahoma 74104, USA.}\ \thanks{\small Department of Mathematics, Aberystwyth University, Ceredigion SY23 3BZ, Wales, UK.}\,\, and Luigi Provenzano\thanks{EPFL, SB Institute of Mathematics, Station 8, CH-1015 Lausanne, Switzerland.}\ \thanks{Corresponding author}}

\date{\ }

\begin{document}

\newcommand{\rea}{\mathbb{R}}

\maketitle

\noindent
{\bf Abstract:}
We consider the spectral problem
\begin{equation*}
\left\{\begin{array}{ll}
-\Delta u_{\varepsilon}=\lambda(\varepsilon)\rho_{\varepsilon}u_{\varepsilon} & {\rm in}\ \Omega\\
\frac{\partial u_{\varepsilon}}{\partial\nu}=0 & {\rm on}\ \partial\Omega
\end{array}\right.
\end{equation*}
in a smooth bounded domain $\Omega$ of $\mathbb R^2$. The  factor $\rho_{\varepsilon}$ which appears in the first equation plays the role of a mass density and it is equal to a constant of order $\varepsilon^{-1}$ in an $\varepsilon$-neighborhood of the boundary and to a constant of order $\varepsilon$ in the rest of $\Omega$.  We study the asymptotic behavior of the eigenvalues $\lambda(\varepsilon)$ and the eigenfunctions $u_{\varepsilon}$ as $\varepsilon$ tends to zero. We obtain explicit formulas for the first and second terms of the corresponding asymptotic expansions by exploiting the solutions of certain auxiliary boundary value problems.
\vspace{11pt}

\vspace{6pt}


\newcommand{\oH}{{\mathaccent'27 H}}


\section{Introduction}

We fix once for all a real number $M>0$  and a bounded connected open set $\Omega$ in $\mathbb R^2$ of class $C^{3}$.
Then, for $\varepsilon>0$ small, we consider the Neumann eigenvalue problem
\begin{equation}\label{Neumann}
\left\{\begin{array}{ll}
-\Delta u_{\varepsilon}=\lambda(\varepsilon)\rho_{\varepsilon}u_{\varepsilon} & {\rm in}\ \Omega,\\
\frac{\partial u_{\varepsilon}}{\partial\nu}=0 & {\rm on}\ \partial\Omega,
\end{array}\right.
\end{equation}
in the unknowns $\lambda(\varepsilon)$ (the eigenvalue) and $u_{\varepsilon}$ (the eigenfunction). The factor $\rho_\varepsilon$ is defined by
\begin{equation*}
\rho_{\varepsilon}(x):=\left\{\begin{array}{ll}
\varepsilon & {\rm in}\ \Omega\setminus\overline\omega_{\varepsilon},\\
\frac{M-\varepsilon |\Omega\setminus\overline\omega_{\varepsilon}|}{|\omega_{\varepsilon}|} & {\rm in}\ \omega_{\varepsilon},
\end{array}\right.
\end{equation*}
where
\begin{equation*}
\omega_{\varepsilon}:=\left\{x\in\Omega:{\rm dist}\left(x,\partial\Omega\right)<\varepsilon\right\}
\end{equation*}
is the strip of width $\varepsilon$ near the boundary $\partial\Omega$ of $\Omega$ (see Figure \ref{strip}). Here and in the sequel $\nu$ denotes the outer unit normal to $\partial\Omega$. 

\begin{figure}
    \centering
       \includegraphics[width=0.4\textwidth]{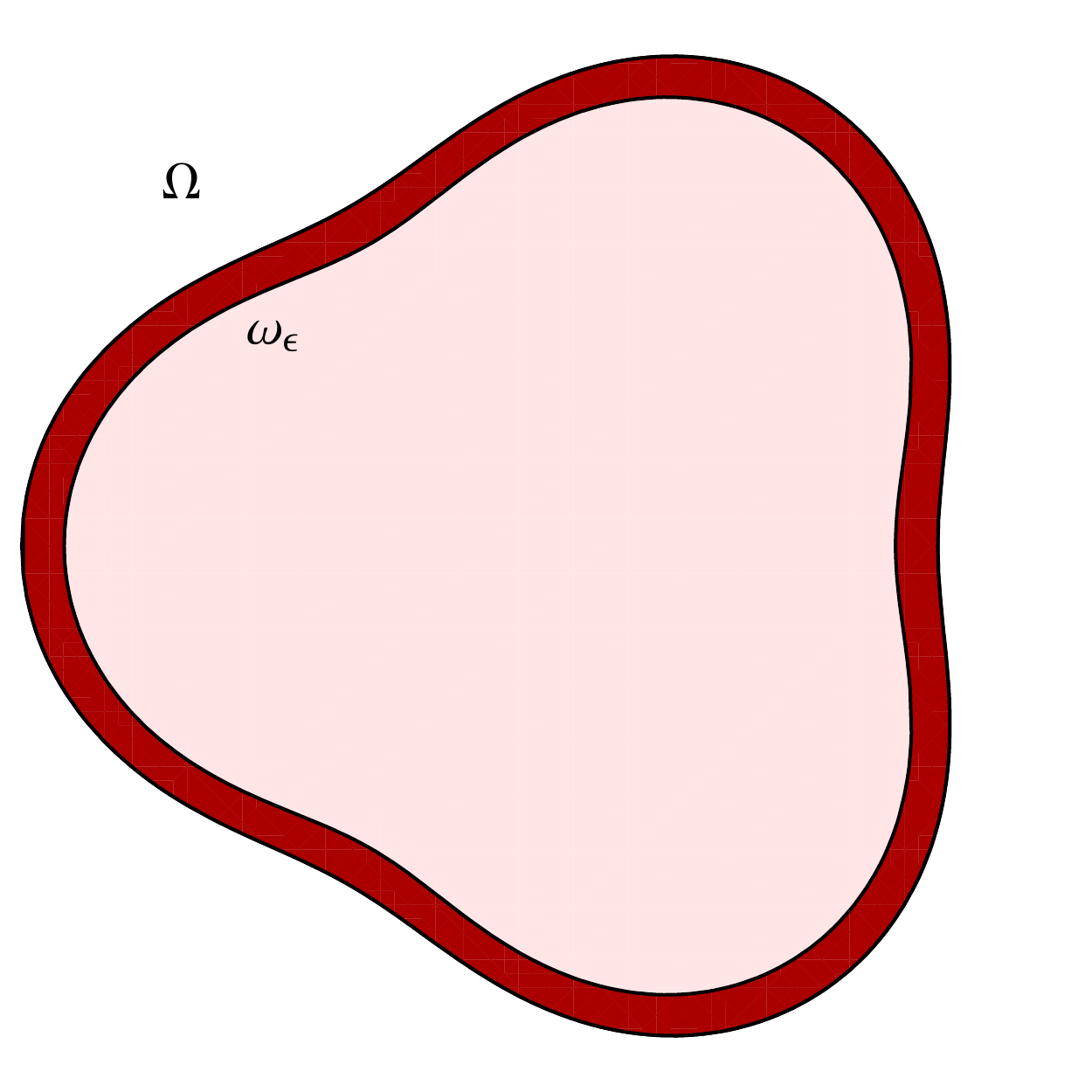}
			\caption{}
			\label{strip}
\end{figure}

 It is well-known that the eigenvalues of \eqref{Neumann} have finite multiplicity and form an increasing sequence
$$
\lambda_0({\varepsilon})<\lambda_1(\varepsilon)\leq\lambda_2(\varepsilon)\leq\cdots\leq\lambda_j(\varepsilon)\leq\cdots\nearrow +\infty.
$$
In addition $\lambda_0(\varepsilon)=0$ and the eigenfunctions corresponding to $\lambda_0(\varepsilon)$ are the constant functions on $\Omega$. We will agree to repeat the eigenvalues according to their multiplicity.

Problem \eqref{Neumann} arises in the study of the transverse vibrations of a thin elastic membrane which occupies at rest the planar domain $\Omega$ (see e.g., \cite{cohil}). The mass of the membrane is distributed accordingly to the density $\rho_{\varepsilon}$. Thus the total mass is given by 
\begin{equation*}
\int_{\Omega}\rho_{\varepsilon}dx=M
\end{equation*}
and it is constant for all $\varepsilon>0$. In particular, most of the mass is concentrated in a $\varepsilon$-neighborhood of the boundary $\partial\Omega$, while the remaining is distributed in the rest of $\Omega$ with a density proportional to $\varepsilon$. The eigenvalues $\lambda_j(\varepsilon)$ are the squares of the natural frequencies of vibration when the boundary of the membrane is left free. The corresponding eigenfunctions represent the profiles of vibration.

Then, we introduce the classical Steklov eigenvalue problem
\begin{equation}\label{Steklov}
\left\{\begin{array}{ll}
\Delta u=0, & {\rm in}\ \Omega,\\
\frac{\partial u}{\partial\nu}=\frac{M}{|\partial\Omega|}\mu u, & {\rm on}\ \partial\Omega,
\end{array}\right.
\end{equation}
in the unknowns $\mu$ (the eigenvalue) and $u$ (the eigenfunction). The spectrum of \eqref{Steklov} consists of an increasing sequence of non-negative eigenvalues of finite multiplicity, which we denote by
$$
\mu_0<\mu_1\leq\mu_{2}\leq\cdots\leq\mu_{j}\leq\cdots\nearrow+\infty.
$$
One easily verifies that $\mu_{0}=0$ and that the corresponding eigenfunctions are the constants functions on $\Omega$. In addition, one can prove that 
for all $j\in\mathbb N$ we have
$$
\lambda_j(\varepsilon)\rightarrow\mu_{j}\quad\text{ as }\varepsilon\rightarrow 0
$$ 
(see, {e.g.}, Arrieta {\it et al.}~\cite{arrieta}, see also Buoso and Provenzano \cite{buosoprovenzano} and Theorem \ref{convergence} here below). Accordingly, one may think to the $\mu_{j}$'s as to the squares of the natural frequencies of vibration of a free elastic membrane with total mass $M$ concentrated on the $1$-dimensional boundary $\partial\Omega$ with constant density $M/{|\partial\Omega|}$. A classical reference for the study of problem \eqref{Steklov} is the paper \cite{steklov} by Steklov. We refer to Girouard and Polterovich \cite{girouardpolterovich} for a recent survey paper and to the recent works of Lamberti and Provenzano \cite{lambertiprovenzano1} and of Lamberti \cite{lamberti1} for related problems. We also refer to  Buoso and Provenzano \cite{buosoprovenzano} for a detailed analysis of the analogous problem for the biharmonic operator.  

The aim of the present paper is to study the asymptotic behavior of the eigenvalues $\lambda_j(\varepsilon)$  of problem \eqref{Neumann} and the corresponding eigenfunctions $u_{j,\varepsilon}$ as $\varepsilon$ goes to zero, i.e., when the thin strip $\omega_{\varepsilon}$ shrinks to the boundary of $\Omega$. To do so, we show the validity of an asymptotic expansion for $\lambda_j(\varepsilon)$ and $u_{j,\varepsilon}$ as $\varepsilon$ goes to zero. In addition, we provide explicit expressions for the first two coefficients in the expansions in terms of solutions of suitable auxiliary problems. In particular, we establish a closed formula for the derivative of $\lambda_j(\varepsilon)$ at $\varepsilon=0$. We observe that such a derivative may be seen as the {\em topological derivative} of $\lambda_j$ for the domain perturbation considered in this paper. We will confine our-selfs to the case when $\lambda_j(\varepsilon)$ converges to a simple eigenvalue $\mu_j$ of \eqref{Steklov}. We observe that such a restriction is justified by the fact that Steklov eigenvalues are generically simple (see e.g., Albert \cite{albert} and Uhlenbeck \cite{uhl}).

As we have written here above, problems \eqref{Neumann} and \eqref{Steklov} concern the elastic behavior of a membrane with mass distributed in a very thin region near the boundary. In view of such a physical interpretation, one may wish to know whether the normal modes of vibration are decreasing or increasing when $\varepsilon>0$ approaches $0$. 

To answer to this question one can compute the value of the derivative of  $\lambda_j(\varepsilon)$ at  $\varepsilon=0$ by exploiting the closed formula that we will obtain. When $\Omega$ is a ball, we can find explicit expressions for the eigenvalues $\lambda_j(\varepsilon)$ and for the corresponding eigenvectors (in this case every eigenvalue is double). In Appendix B we have verified that in such special case the eigenvalues are locally decreasing when $\varepsilon$ approaches $0$ from above. Accordingly the Steklov eigenvalues of the ball are local minimizers of the $\lambda_j(\varepsilon)$. This result is in agreement with the value of the derivative of  $\lambda_j(\varepsilon)$ at  $\varepsilon=0$  that one may compute from our closed formula obtained for a general domain $\Omega$ of class $C^3$.

We observe here that asymptotics for vibrating systems (membranes or bodies) containing masses along curves or masses concentrated at certain points have been considered by several authors in the last decades (see, {\it e.g.},  Golovaty {\it et al.}~\cite{gol1},  Lobo and P\'erez \cite{lope1} and Tchatat \cite{tchatatbook}). We also refer to  Lobo and P{\'e}rez \cite{loboperez1,loboperez2} where the authors consider the vibration of membranes and bodies carrying concentrated masses near the boundary, and to Golovaty {\it et al.}~\cite{gomezpereznazarov2,gomezpereznazarov1}, where the authors consider spectral stiff problems in domains surrounded by thin bands. Let us recall that these problems have been addressed also for vibrating plates (see Golovaty {\it et al.}~\cite{golonape_plates1,golonape_plates2} and the references therein). We also mention the alternative approach based on potential theory and functional analysis proposed in Musolino and Dalla Riva \cite{musolinodallariva} and Lanza de Cristoforis \cite{lanza}.

The paper is organized as follows. In Section 2 we introduce the notation and certain preliminary tools that are used through the paper. In Section \ref{sec:3} we state our main Theorems \ref{asymptotic_eigenvalues} and \ref{asymptotic_eigenfunctions}, which concern the asymptotic expansions of the eigenvalues and of the eigenfunctions of \eqref{Neumann}, respectively. In Theorem \ref{asymptotic_eigenvalues} we also provide the explicit formula for the topological derivative of the eigenvalues of \eqref{Neumann}. The proof of Theorems \ref{asymptotic_eigenvalues} and \ref{asymptotic_eigenfunctions} is presented in Sections \ref{sec:4} and \ref{sec:5}. In Section \ref{sec:4} we justify the asymptotic expansions of Theorems \ref{asymptotic_eigenvalues} and \ref{asymptotic_eigenfunctions} up to the zero order terms. Then in Section \ref{sec:5} we justify the asymptotic expansions up to the first order terms {and, as a byproduct, we prove the validity of the formula for the topological derivative.} At the end of the paper we have included two Appendices. In the Appendix A we consider an auxiliary problem and prove its well-posedness. In the last Appendix B we consider the case when $\Omega$ is the unit ball and prove that the Steklov eigenvalues are local minimizers of the Neumann eigenvalues for $\varepsilon$ small enough.


\section{Preliminaries}\label{sec:2}

\subsection{A convenient change of variables}

Since $\Omega$ is of class $C^{3}$, it is well-known that there exists ${\varepsilon'_\Omega}>0$ such that the map $x\mapsto x-\varepsilon\nu(x)$ is a diffeomorphism {of class $C^2$} from $\partial\Omega$ to $\partial\omega_{\varepsilon}\cap\Omega$ for all $\varepsilon\in(0,{\varepsilon'_\Omega})$. We will exploit this fact to introduce curvilinear coordinates in the strip $\omega_{\varepsilon}$. To do so, we denote by $\gamma:[0,|\partial\Omega|)\rightarrow\partial\Omega$ the arc length parametrization of the boundary $\partial\Omega$. Then, one verifies that the map $\psi:[0,|\partial\Omega|)\times(0,\varepsilon)\rightarrow \omega_{\varepsilon}$ defined by $\psi(s,t):=\gamma(s)-t\nu(\gamma(s))$, for all $(s,t)\in [0,|\partial\Omega|)\times(0,\varepsilon)$, is a diffeomorphism and we can use the curvilinear coordinates $(s,t)$ in the strip $\omega_\varepsilon$.
We denote by $\kappa(s)$ the signed curvature of $\partial\Omega$, namely we set $\kappa(s)=\gamma_1'(s)\gamma_2''(s)-\gamma_2'(s)\gamma_1''(s)$ for all $s\in [0,|\partial\Omega|)$. 

In order to study problem \eqref{Neumann} it is also convenient to introduce a change of variables by setting $\xi=t/{\varepsilon}$. Accordingly, we denote by $\psi_{\varepsilon}$ the function from $[0,|\partial\Omega|)\times(0,1)$ to $\omega_{\varepsilon}$ defined by $\psi_{\varepsilon}(s,\xi):=\gamma(s)-\varepsilon\xi\nu(\gamma(s))$ for all $(s,\xi)\in[0,|\partial\Omega|)\times(0,1)$. The variable $\xi$ is usually called `rapid variable'. We observe that in this new system of coordinates $(s,\xi)$, the strip $\omega_{\varepsilon}$ is transformed into a band of length $|\partial\Omega|$ and width $1$ (see Figures \ref{fig43} and \ref{fig44}). Moreover, we note that if $\varepsilon<({\sup_{s\in[0,|\partial\Omega|)}|\kappa(s)|})^{-1}$, then we have 
\begin{equation}\label{positiveinf}
\inf_{(s,\xi)\in[0,|\partial\Omega|)\times(0,1)}\bigl(1-\varepsilon\xi\kappa(s)\bigr)>0\,,
\end{equation} 
so that $|\det D\psi_{\varepsilon}|=\varepsilon(1-\varepsilon\xi\kappa(s))$ for all $(s,\xi)\in[0,|\partial\Omega|)\times(0,1)$.

We will also need to write the gradient of a function $u$ on $\omega_\varepsilon$ with respect to the coordinates $(s,\xi)$. To do so we take 
\[
\varepsilon''_\Omega:=\min\left\{\varepsilon'_\Omega\,,\, \biggl({\sup_{s\in[0,|\partial\Omega|)}|\kappa(s)|}\biggr)^{-1}\right\}.
\]
and we consider $\varepsilon\in(0,\varepsilon''_\Omega)$. Then we have
\begin{equation*}
\left(\nabla u\circ\psi_{\varepsilon}\right)(s,\xi)=\left(\begin{array}{ll}{\frac{\gamma_1'(s)}{1-\varepsilon\xi\kappa(s)}\partial_s(u\circ\psi_{\varepsilon}(s,\xi))-\frac{\gamma_2'(s)+\varepsilon\xi\gamma_1''(s)}{\varepsilon(1-\varepsilon\xi\kappa(s))}\partial_{\xi}(u\circ\psi_{\varepsilon}(s,\xi))}\\{\frac{\gamma_2'(s)}{1-\varepsilon\xi\kappa(s)}\partial_s(u\circ\psi_{\varepsilon}(s,\xi))+\frac{\gamma_1'(s)-\varepsilon\xi\gamma_2''(s)}{\varepsilon(1-\varepsilon\xi\kappa(s))}\partial_{\xi}(u\circ\psi_{\varepsilon}(s,\xi){)}}\end{array}\right)
\end{equation*}
and therefore

\begin{equation}\label{grad2}
\begin{split}
&\left(\nabla u\circ\psi_{\varepsilon}\cdot\nabla v\circ\psi_{\varepsilon}\right)(s,\xi)\\
&\qquad=\frac{1}{\varepsilon^2}\partial_{\xi}(u\circ\psi_{\varepsilon}(s,\xi))\partial_{\xi}(v\circ\psi_{\varepsilon}(s,\xi))+\frac{\partial_s (u\circ\psi_{\varepsilon}(s,\xi))\partial_s (v\circ\psi_{\varepsilon}(s,\xi))}{(1-\varepsilon\xi\kappa(s))^2}
\end{split}
\end{equation}
for all $(s,\xi)\in[0,|\partial\Omega|)\times(0,1)$.

\begin{figure}[!ht]
 \centering
    \includegraphics[width=0.6\textwidth]{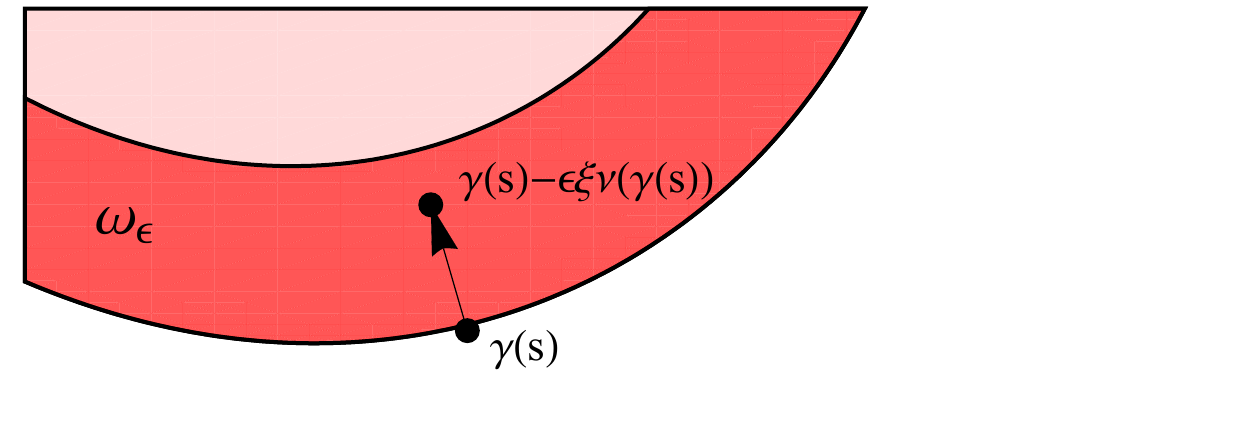}
		 \caption{}\label{fig43}
\end{figure}
\begin{figure}[!ht]
  \centering
      \includegraphics[width=0.6\textwidth]{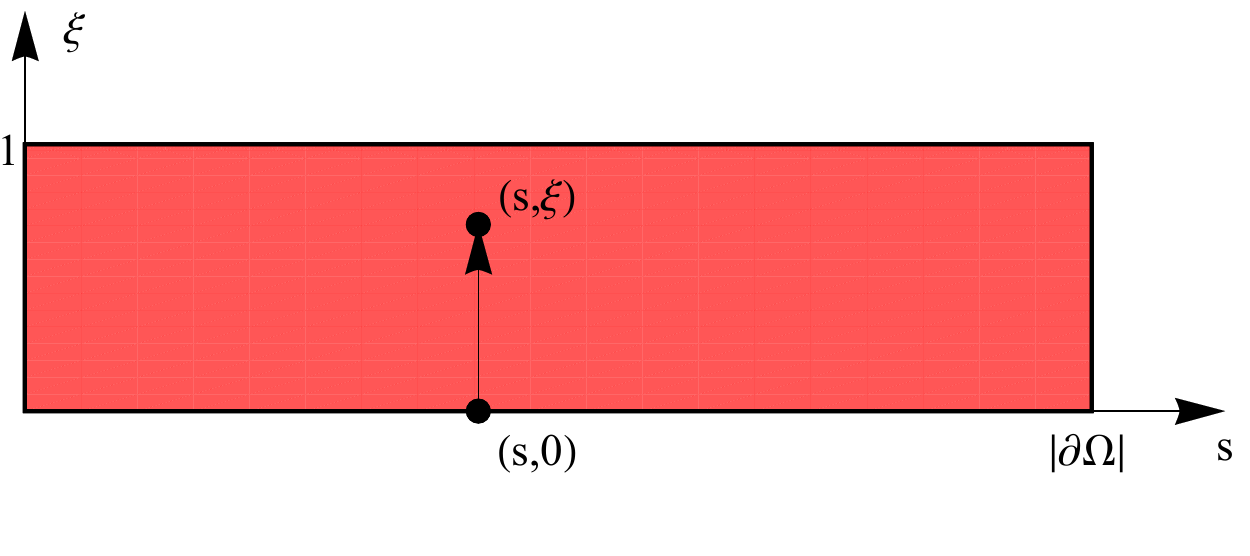}
\caption{}\label{fig44}
\end{figure}

\subsection{Some remarks about $\rho_\varepsilon$}

We can write $\rho_{\varepsilon}=\varepsilon+\frac{1}{\varepsilon}\tilde\rho_{\varepsilon}\chi_{\omega_{\varepsilon}}$, where $\chi_{\omega_{\varepsilon}}$ is the characteristic function of $\omega_{\varepsilon}$ and 
\begin{equation}\label{def_tilde_rho}
\tilde\rho_{\varepsilon}:=\varepsilon\left(\frac{M-\varepsilon |\Omega\setminus\overline\omega_{\varepsilon}|}{|\omega_{\varepsilon}|}\right)-\varepsilon^2.
\end{equation}
Then we observe that  for $\varepsilon\in(0,\varepsilon''_\Omega)$ we have
\begin{equation}\label{oe}
|\omega_{\varepsilon}|=\varepsilon |\partial\Omega|-\frac{\varepsilon^2}{2}K\,,
\end{equation}
where $K$ is defined by
\begin{equation}\label{K}
K:=\int_0^{|\partial\Omega|}\kappa(s)ds.
\end{equation}
By \eqref{positiveinf} it follows that $|\partial\Omega|-\frac{\varepsilon}{2}K>0$. Then by \eqref{def_tilde_rho} and \eqref{oe} one verifies that there exists a real analytic map $\tilde{R}$ from $(-\varepsilon''_\Omega,\varepsilon''_\Omega)$ to $\mathbb{R}$ such that 
\begin{equation}\label{asymptotic_rho}
\tilde\rho_{\varepsilon}=\frac{M}{|\partial\Omega|}+\frac{\frac{1}{2}K M-|\Omega||\partial\Omega|}{|\partial\Omega|^2}\varepsilon+\varepsilon^2\tilde{R}(\varepsilon)\qquad\forall\varepsilon\in(0,\varepsilon''_\Omega)\,.
\end{equation}
We are now legitimate to fix once for all a real number 
\begin{equation}\label{eO}
\text{$\varepsilon_\Omega\in (0,\varepsilon''_\Omega)$ such that $\inf_{\varepsilon\in(0,\varepsilon_\Omega)}\tilde\rho_{\varepsilon}>0.$}
\end{equation}

\subsection{Weak formulation of problem (\ref{Neumann}) and the resolvent operator $\mathcal A_\varepsilon$}

For all $\varepsilon\in(0,\varepsilon_\Omega)$, we denote by $\mathcal H_{\varepsilon}(\Omega)$ the Hilbert space consisting of the functions in the standard Sobolev space $H^1(\Omega)$ endowed with the bilinear form
\begin{equation}\label{bilinear_eps}
\left\langle u,v\right\rangle_{\varepsilon}:=\int_{\Omega}\nabla u\cdot\nabla v dx+\int_{\Omega}\rho_{\varepsilon}u v dx\ \ \forall u,v\in \mathcal H_{\varepsilon}(\Omega).
\end{equation}
The bilinear form \eqref{bilinear_eps} induces on $H^1(\Omega)$ a norm which is equivalent to the standard one. We denote such a norm by $\|\cdot\|_{\varepsilon}$.

We note that the  weak formulation of problem \eqref{Neumann} can be stated as follows: a pair $(\lambda(\varepsilon), u_{\varepsilon})\in \mathbb R\times H^1(\Omega)$ is a solution of \eqref{Neumann} in the weak sense if and only if 
\begin{equation*}
\int_{\Omega}\nabla u_{\varepsilon}\cdot\nabla\varphi dx=\lambda(\varepsilon)\int_{\Omega}\rho_{\varepsilon} u_{\varepsilon}\varphi dx\ \ \forall \varphi\in H^1(\Omega).
\end{equation*}
Then, for all $\varepsilon\in(0,\varepsilon_\Omega)$ we introduce the linear operator $\mathcal A_{\varepsilon}$ from $\mathcal H_{\varepsilon}(\Omega)$ to itself which maps a function $f\in\mathcal H_{\varepsilon}(\Omega)$ to the function $u\in\mathcal H_{\varepsilon}$ such that
\begin{equation}\label{A_eps}
\int_{\Omega}\nabla u\cdot\nabla\varphi dx+\int_{\Omega}\rho_{\varepsilon} u\varphi dx=\int_{\Omega}\rho_{\varepsilon}f\varphi dx\ \ \forall\varphi\in\mathcal H_{\varepsilon}(\Omega).
\end{equation}
{We note that such a function $u\in\mathcal H_{\varepsilon}(\Omega)$ exists by the Riesz representation theorem and it is unique because $\int_{\Omega}\nabla u\cdot\nabla\varphi dx+\int_{\Omega}\rho_{\varepsilon} u\varphi dx=0$ for all $\varphi\in\mathcal H_{\varepsilon}(\Omega)$ implies that $\norm{u}_\varepsilon=0$. }

\medskip

In the sequel we {will heavily} exploit the following lemma. We refer to Ole{\u\i}nik {\it et al.}~\cite[III.1]{oleinik} for its proof.
\begin{lemma}\label{lemma_fondamentale}
Let {$A$ be a   compact, self-adjoint and positive linear operator} from a separable Hilbert space $H$ to itself. Let $u\in H$, with $\|u\|_H=1$. Let $\eta,r>0$ be such that $\|A u-\eta u\|_H\leq r$. Then, there exists an eigenvalue $\eta^*$ of the operator $A$ which {satisfies} the inequality $|\eta-\eta^*|\leq r$. Moreover, for any $r^*>r$ there exists $u^*\in H$ with $\|u^*\|_H=1$, $u^*$ belonging to the space generated by all the eigenfunctions associated with an eigenvalue of the operator $A$ lying on the segment $[\eta-r^*,\eta+r^*]$, and such that 
$$
\|u-u^*\|_H\leq\frac{2 r}{r^*}.
$$
\end{lemma}

We observe that the operator $\mathcal{A}_\varepsilon$ is a good candidate for the application of Lemma \ref{lemma_fondamentale}. Indeed, we have the following Proposition \ref{Ae}.

\begin{proposition}\label{Ae}
For all $\varepsilon\in(0,\varepsilon_\Omega)$ the map $\mathcal{A}_\varepsilon$ is a compact, self-adjoint and positive linear operator from $\mathcal H_{\varepsilon}$ to itself.
\end{proposition}
\proof The proof that $\mathcal{A}_\varepsilon$ is self-adjoint and positive can be effected by noting that $\langle\mathcal{A}_\varepsilon f,g\rangle_\varepsilon=\int_{\Omega}\rho_{\varepsilon}fg\, dx$ for all $f,g\in\mathcal H_{\varepsilon}(\Omega)$. To prove that $\mathcal{A}_\varepsilon$ is compact we denote by $\tilde{\mathcal{A}}_\varepsilon$ the linear operator from $L^2(\Omega)$ to $\mathcal H_{\varepsilon}(\Omega)$ which takes a function $f\in L^2(\Omega)$ to the unique element $u\in \mathcal H_{\varepsilon}(\Omega)$ which satisfies the condition in \eqref{A_eps}. By the Riesz representation theorem one verifies that $\tilde{\mathcal{A}}_\varepsilon$ is well defined.  In addition, we can prove that  $\tilde{\mathcal{A}}_\varepsilon$ is bounded. Indeed, we have
\[
\|\tilde{\mathcal{A}}_\varepsilon f\|_\varepsilon^2=\langle\tilde{\mathcal{A}}_\varepsilon f,\tilde{\mathcal{A}}_\varepsilon f\rangle_\varepsilon=\int_{\Omega}\rho_{\varepsilon}f\tilde{\mathcal{A}}_\varepsilon f\, dx
\]
and by a computation based on the H\"older inequality one verifies that 
\[
\int_{\Omega}\rho_{\varepsilon}f\tilde{\mathcal{A}}_\varepsilon f\, dx=\int_{\Omega}\rho^{\frac{1}{2}}_{\varepsilon}f\rho^{\frac{1}{2}}_{\varepsilon}\tilde{\mathcal{A}}_\varepsilon f\, dx \le \left(\int_{\Omega}\rho_{\varepsilon}f^2dx\right)^{\frac{1}{2}}\left(\int_{\Omega}\rho_{\varepsilon}(\tilde{\mathcal{A}}_\varepsilon f)^2dx\right)^{\frac{1}{2}}\,,
\] 
which implies that $\|{\tilde{\mathcal{A}}_\varepsilon f}\|_\varepsilon\le (\sup_{x\in\Omega}{\rho_\varepsilon(x)}^{\frac{1}{2}})\norm{f}_{L^2(\Omega)}$ for all $f\in L^2(\Omega)$. Then we denote by $\mathcal{E}_\varepsilon$ the embedding map from $\mathcal{H}_\varepsilon(\Omega)$ to $L^2(\Omega)$. Via the natural isomorphism from  $\mathcal{H}_\varepsilon(\Omega)$ to $H^1(\Omega)$, one deduces that $\mathcal{E}_\varepsilon$ is compact. Since $\mathcal{A}_\varepsilon=\tilde{\mathcal{A}}_\varepsilon\circ\mathcal{E}_\varepsilon$, we conclude that also  $\mathcal{A}_\varepsilon$ is compact.
\qed

\medskip

We conclude this subsection by observing that the $L^2(\Omega)$ norm of a function in $H^1(\Omega)$ is uniformly bounded by its $\|\cdot\|_\varepsilon$ norm for all $\varepsilon\in(0,\varepsilon_\Omega)$. We will prove such a result in Proposition \ref{L2<eps} below by exploiting the following Lemma \ref{poinc_peso}.

\begin{lemma}\label{poinc_peso}
There exists $C_{\Omega}>0$ such that
\[
\left\|u-\frac{1}{M}\int_{\Omega}\rho_{\varepsilon}u dx\right\|_{L^2(\Omega)}\leq C_{\Omega}\|\nabla u\|_{L^2(\Omega)}, 
\]
for all $u\in H^1(\Omega)$ and for all $\varepsilon\in(0,\varepsilon_\Omega)$.
\proof
 We argue by contradiction and we assume that   there exist a sequence $\{\tau_k\}_{k\in\mathbb{N}}\subset(0,\varepsilon_\Omega)$ and a sequence $\{w_k\}_{k\in\mathbb{N}}\subset H^1(\Omega)$ such that
\begin{equation}\label{assurdo}
\left\|w_k-\frac{1}{M}\int_{\Omega}\rho_{\tau_k}w_k dx\right\|_{L^2(\Omega)}> k\|\nabla w_k\|_{L^2(\Omega)}
\end{equation}
for all $k\in\mathbb{N}$. Since  $\{\tau_k\}_{k\in\mathbb{N}}$ is bounded there exist ${\tau}\in [0,\varepsilon_\Omega]$ and a subsequence of $\{\tau_k\}_{k\in\mathbb{N}}$, which we still denote by $\{\tau_k\}_{k\in\mathbb{N}}$, such that   $\tau_k\to{\tau}$ as $k\to\infty$. Then we set
$$
v_k:=\left\|w_k-\frac{1}{M}\int_{\Omega}\rho_{\tau_k}w_k dx\right\|_{L^2(\Omega)}^{-1}\left(w_k-\frac{1}{M}\int_{\Omega}\rho_{\tau_k}w_k dx\right)
$$
for all $k\in\mathbb{N}$.
We verify that $\int_{\Omega}\rho_{\tau_k}v_k dx=0$ and $\|v_k\|_{L^2(\Omega)}=1$ for all $k\in\mathbb N$, and from \eqref{assurdo}, $\|\nabla v_k\|<\frac{1}{k}$. Then $v_k$ is bounded in $H^1(\Omega)$ and we can extract a subsequence, which we still denote by $\lbrace v_k\rbrace_{k\in\mathbb N}$, such that $v_k\rightharpoonup v$ weakly in $H^1(\Omega)$ and $v_k\rightarrow v$ strongly in $L^2(\Omega)$, for some $v\in H^1(\Omega)$. Moreover, since $\|\nabla v_k\|_{L^2(\Omega)}<\frac{1}{k}$ one can verify that $\nabla v=0$ a.e.~in $\Omega${, and thus $v$ is constant on $\Omega$. In addition, 
\begin{equation}\label{poic_peso.eq1}
\|v\|_{L^2(\Omega)}=\lim_{k\to\infty}\|v_k\|_{L^2(\Omega)}=1.
\end{equation}
We now prove that \eqref{poic_peso.eq1} leads to a contradiction. Indeed, we can prove that $v=0$. We consider separately the case when ${\tau}>0$ and the case when ${\tau}=0$. For ${\tau}>0$ we verify that  $\lim_{k\rightarrow\infty}\int_{\Omega}\rho_{\tau_k}v_k dx=\int_{\Omega}\rho_{{\tau}}v dx$. Then $\int_{\Omega}\rho_{{\tau}}v dx=0$, because
 $\int_{\Omega}\rho_{\tau_k}v_k dx=0$.  Since $\rho_{{\tau}}>0$ and $v$ is constant, it follows that $v=0$.  If instead ${\tau}=0$, then, by an argument based on \cite[Lemmas 3.1.22, 3.1.28]{phd} we have $\lim_{k\rightarrow\infty}\int_{\Omega}\rho_{\tau_k}v_k dx=\frac{M}{|\partial\Omega|}\int_{\partial\Omega}v d\sigma$. 
Since $\int_{\Omega}\rho_{\tau_k}v_k dx=0$, it follows that $\int_{\partial\Omega}v d\sigma=0$. Since $v$ is constant on $\Omega$, we deduce that $v=0$.}
\endproof
\end{lemma}

We are now ready to prove Proposition \ref{L2<eps}.

\begin{proposition}\label{L2<eps} 
If $\varepsilon\in(0,\varepsilon_\Omega)$ and $v\in H^1(\Omega)$, then 
$$
\|v\|_{L^2(\Omega)}\le \max\left\{C_{\Omega},\sqrt\frac{|\Omega|}{M}\right\}\|v\|_{\varepsilon},$$
where $C_\Omega$ is the constant which appears in Lemma \ref{poinc_peso}.
\end{proposition} 
\begin{proof}
First we observe that  
\begin{equation}\label{J1e_eq1}
\int_{\Omega}\rho_{\varepsilon} v dx=\int_{\Omega}\rho_{\varepsilon}^{\frac{1}{2}}\rho_{\varepsilon}^{\frac{1}{2}}v dx\leq \left(\int_{\Omega}\rho_{\varepsilon}dx\right)^{\frac{1}{2}}\left(\int_{\Omega}\rho_{\varepsilon}v^2 dx\right)^{\frac{1}{2}}=M^{\frac{1}{2}}\left(\int_{\Omega}\rho_{\varepsilon}v^2 dx\right)^{\frac{1}{2}}.
\end{equation}
Then, by  Lemma \ref{poinc_peso} and by \eqref{J1e_eq1} we deduce that
\begin{equation*}
\begin{split}
\|v\|_{L^2(\Omega)}&=\norm{v-\frac{1}{M}\int_{\Omega}\rho_{\varepsilon} v dx+\frac{1}{M}\int_{\Omega}\rho_{\varepsilon} v dx}_{L^2(\Omega)}\\
& \leq\norm{v-\frac{1}{M}\int_{\Omega}\rho_{\varepsilon} v dx}_{L^2(\Omega)}+\norm{\frac{1}{M}\int_{\Omega}\rho_{\varepsilon} v dx}_{L^2(\Omega)}\\
& \leq C_{\Omega}\|\nabla v\|_{L^2(\Omega)}+\sqrt\frac{|\Omega|}{M}\left(\int_{\Omega}\rho_{\varepsilon}v^2 dx\right)^{\frac{1}{2}}\,.
\end{split}
\end{equation*} 
Now the validity of the proposition follows by a straightforward computation. 
\end{proof}

\subsection{Known results on the limit behavior of $\lambda_j(\varepsilon)$}

In  the following Theorem \ref{convergence} we recall some results on the {limit} behavior of the eigenelements of problem \eqref{Neumann}.
\begin{theorem}\label{convergence}
The following statements hold.
\begin{enumerate}
\item[(i)]
For all $j\in\mathbb N$ it holds
$$
\lim_{\varepsilon\rightarrow 0}\lambda_j(\varepsilon)=\mu_{j}.
$$
\item[(ii)] Let $\mu_j$ be a simple eigenvalue of problem \eqref{Steklov} and let $\lambda_j(\varepsilon)$ be such that $\lim_{\varepsilon\rightarrow 0}\lambda_j(\varepsilon)=\mu_{j}$. Then there exists $\varepsilon_j>0$ such that $\lambda_j(\varepsilon)$ is simple for all $\varepsilon\in(0,\varepsilon_j)$.
\end{enumerate}
\end{theorem}
The proof of Theorem \ref{convergence} can be carried out  by using the notion of compact convergence for the resolvent operators, and can also be obtained as a consequence of the more general results proved  in Arrieta {\it et al.}~\cite{arrieta} (see also Buoso and Provenzano \cite{buosoprovenzano}). 

From Theorem \ref{convergence}, it follows that the function $\lambda_j(\cdot)$ which takes $\varepsilon>0$ to $\lambda_j(\varepsilon)$ can be extended with continuity at $\varepsilon =0$ by setting   $\lambda_j(0):=\mu_{j}$ for all $j\in {\mathbb{N}}$.


\section{Description of the main results}\label{sec:3}

In this section we state our main Theorems \ref{asymptotic_eigenvalues} and \ref{asymptotic_eigenfunctions} which will be proved in Sections \ref{sec:4} and \ref{sec:5} below. We will use the following notation: if $j\in\mathbb N$ and $\mu_j$ is a simple eigenvalue of problem \eqref{Steklov}, then we take 
\[
\varepsilon_{\Omega,j}:=\min\{\varepsilon_j\,,\,\varepsilon_\Omega\}
\] 
with $\varepsilon_j$ as in Theorem \ref{convergence} and $\varepsilon_\Omega$ as in \eqref{eO},  so that $\lambda_j(\varepsilon)$ is a simple eigenvalue of \eqref{Neumann} for all $\varepsilon\in(0,\varepsilon_{\Omega,j})$. If $f$ is an invertible function, than  $f^{(-1)}$ denotes the inverse of $f$, as opposed to $r^{-1}$ and $f^{-1}$ which denote the reciprocal of a real non-zero number or of a non-vanishing function.

In the following Theorem \ref{asymptotic_eigenvalues} we provide an asymptotic expansion of the eigenvalue $\lambda_j(\varepsilon)$ up to a remainder of order $\varepsilon^2$. 

\begin{theorem}\label{asymptotic_eigenvalues}
Let $j\in\mathbb N$. Assume that $\mu_j$ is a simple eigenvalue of problem \eqref{Steklov}. Then
\begin{equation}\label{expansion_eigenvalues}
\lambda_j(\varepsilon)=\mu_j+\varepsilon\mu_j^1+O(\varepsilon^2)\quad\text{as }\varepsilon\rightarrow 0
\end{equation}
where 
\begin{equation}\label{top_der_formula}
\mu_j^1=\frac{|\Omega|\mu_j}{M}-\frac{|\partial\Omega|\mu_j}{M}\int_{\Omega}u_j^2dx+\frac{2M\mu_j^2}{3|\partial\Omega|}+\frac{\mu_j}{2}\int_{\partial\Omega}u_j^2\kappa{\circ\gamma^{(-1)}} d\sigma-\frac{K\mu_j}{2|\partial\Omega|}.
\end{equation}
The constant $K$ is given by \eqref{K} and $u_j\in H^1(\Omega)$ is the unique eigenfunction of problem \eqref{Steklov} associated with the eigenvalue $\mu_j$ {which satisfies the additional condition} 
\begin{equation}\label{ucondition}
\int_{\partial\Omega}u_j^2 d\sigma=1.
\end{equation}

\end{theorem}
 
 In Theorem \ref{asymptotic_eigenfunctions} here below we show an asymptotic expansion for the eigenfunction $u_{j,\varepsilon}$ associated to $\lambda_j(\varepsilon)$.

\begin{theorem}\label{asymptotic_eigenfunctions}
Let $j\in\mathbb N$ and assume that $\mu_j$ is a simple eigenvalue of problem \eqref{Steklov}. {Let $0<\varepsilon_{\Omega,j}<\varepsilon_\Omega$ be} such that $\lambda_j(\varepsilon)$ is a simple eigenvalue of problem \eqref{Neumann}  for all $\varepsilon\in(0,\varepsilon_{\Omega,j})$. Let $u_j$ be the unique eigenfunction of problem \eqref{Steklov} associated with $\mu_j$ {which satisfies the additional condition \eqref{ucondition}.}
For all $\varepsilon\in (0,\varepsilon_{\Omega,j})$, let $u_{j,\varepsilon}$ be the unique eigenfunction of problem \eqref{Neumann} corresponding to $\lambda_j(\varepsilon)$ {which satisfies the additional condition  
\begin{equation}\label{uecondition}
\frac{|\partial\Omega|}{M}\int_{\Omega}\rho_{\varepsilon}u_{j,\varepsilon}^2dx=1\,.
\end{equation}}
Then there exist $u_j^1\in H^1(\Omega)$ and $w_j\in H^1([0,|\partial\Omega|)\times(0,1))$ such that
\begin{equation}\label{expansion_eigenfunctions}
u_{j,\varepsilon}=u_j+\varepsilon u_j^1+\varepsilon v_{j,\varepsilon}+O(\varepsilon^2)\quad{\text{in }\ L^2(\Omega)\text{ as }\varepsilon\rightarrow 0},
\end{equation}
where the function $v_{j,\varepsilon}\in H^1(\Omega)$ is the extension by $0$ of $w_j\circ\psi_{\varepsilon}^{(-1)}$ to $\Omega$.
\end{theorem}
We shall present explicit formulas for $w_j$ {in terms of $\mu_j$ and $u_j$} (see formula \eqref{w0}) and we shall identify $u_j^1$ as the solution to a certain boundary value problem (see problem \eqref{u1_problem_simple_true}). {We also note that $\norm{v_{j,\varepsilon}}_{L^2(\Omega)}\in O(\sqrt{\varepsilon})$, so that the third term in \eqref{expansion_eigenfunctions} is in $O(\varepsilon^{\frac{3}{2}})$ in $L^2(\Omega)$ (cf. Proposition \ref{vjeL2}).}

The proof of Theorems \ref{asymptotic_eigenvalues} and \ref{asymptotic_eigenfunctions} consists of two steps. In the first step (Section \ref{sec:4}) we show that the quantity $\lambda_j(\varepsilon)-\mu_j$ is of order $\varepsilon$ as $\varepsilon$ tends to zero. Moreover,  we introduce the function $w_j$ and we show that {$\|u_{j,\varepsilon}-u_j\|_{L^2(\Omega)}$} is of order $\varepsilon$ as $\varepsilon$ tends to zero. In the second step (Section \ref{sec:5}) we complete the proof of Theorems \ref{asymptotic_eigenvalues} and \ref{asymptotic_eigenfunctions} {by proving the validity of \eqref{expansion_eigenvalues} and \eqref{expansion_eigenfunctions}} and we introduce the boundary value problem which identifies $u_j^1$.


\section{First step}\label{sec:4}

We begin here the proof of Theorems \ref{asymptotic_eigenvalues} and \ref{asymptotic_eigenfunctions}. Accordingly, we fix $j\in\mathbb N$ and we take $\mu_j$, $u_j$, $\varepsilon_{\Omega,j}$,  $\lambda_j(\varepsilon)$, and $u_{j,\varepsilon}$ as in the statements of Theorems \ref{asymptotic_eigenvalues} and \ref{asymptotic_eigenfunctions}.  The  aim of this section is to prove the following intermediate result.
 \begin{proposition}\label{intermediate}
We have
\begin{equation}\label{intermediate.eq1}
\text{$\lambda_j(\varepsilon)=\mu_j+O(\varepsilon)$  as $\varepsilon\to 0$}
\end{equation}
and 
\begin{equation}\label{intermediate.eq2}
\text{$u_{j,\varepsilon}=u_j+O(\varepsilon)$ in $L^2(\Omega)$ as $\varepsilon\to 0$.}
\end{equation}
\end{proposition}
In other words, we wish to justify the expansions \eqref{expansion_eigenvalues} and \eqref{expansion_eigenfunctions} up to a remainder of order $\varepsilon$. (We observe here that Theorem \ref{convergence} states the convergence of $\lambda_j(\varepsilon)$ to $\mu_j$, but it does not provide any information on the rate of convergence.)

We introduce the following notation.  We denote by $w_{j}$ the function  from $[0,|\partial\Omega|)\times[0,1]$ to $\mathbb R$ defined by
\begin{equation}\label{w0}
w_{j}(s,\xi):=-\frac{M\mu_j}{2|\partial\Omega|}(u_j\circ\gamma(s))\left(\xi-1\right)^2\quad\forall (s,\xi)\in[0,|\partial\Omega|)\times[0,1]\,.
\end{equation}
By a straightforward computation one verifies that $w_{j}$ solves the following problem
\begin{equation}\label{w0probl}
\left\{\begin{array}{ll}
-\partial^2_{\xi}w_{j}(s,\xi)=\frac{M\mu_j}{|\partial\Omega|}(u_j\circ\gamma(s)), & (s,\xi)\in [0,|\partial\Omega|)\times(0,1),\\
\partial_{\xi}w_{j}(s,0)=\frac{M\mu_j}{|\partial\Omega|}(u_j\circ\gamma(s)), & s\in[0,|\partial\Omega|),\\
w_j(s,1)=\partial_{\xi}w_{j}(s,1)=0, & s\in [0,|\partial\Omega|).
\end{array}
\right.
\end{equation}
Then for all $\varepsilon\in(0,\varepsilon_{\Omega,j})$ we denote by $v_{j,\varepsilon}\in H^1(\Omega)$ the extensions by $0$ of $w_{j}\circ\psi_{\varepsilon}^{(-1)}$ to $\Omega$. We note that by construction $v_{j,\varepsilon}\in H^1(\Omega)$. We also observe that the $L^2(\Omega)$ norm of $v_{j,\varepsilon}$ is in $O(\sqrt\varepsilon)$ as $\varepsilon\to 0$. Indeed, we have the following proposition. 

\begin{proposition}\label{vjeL2} There is a constant $C>0$ such that $\|v_{j,\varepsilon}\|_{L^2(\Omega)}\le C\sqrt\varepsilon$ for all $\varepsilon\in(0,\varepsilon_{\Omega,j})$.
\end{proposition}
\begin{proof}
Since $v_{j,\varepsilon}$ is  the extensions by $0$ of $w_{j}\circ\psi_{\varepsilon}^{(-1)}$ to $\Omega$, by the rule of change of variables in integrals we have
\[
\begin{split}
\int_{\omega_{\varepsilon}}v_{j,\varepsilon}^2 dx&=\varepsilon\int_0^{|\partial\Omega|}\int_0^1w^2_j(s,\xi)(1-\varepsilon\xi\kappa(s))\,d\xi ds\\
&\le \left(\norm{w_j}_{L^2([0,\partial\Omega)\times(0,1))}\sup_{(s,\xi)\in[0,\partial\Omega)\times(0,\varepsilon)}|1-\varepsilon\xi\kappa(s)|\right)\varepsilon\,.
\end{split}
\]
\end{proof}

We also observe that $\sqrt\varepsilon\norm{v_{j,\varepsilon}}_\varepsilon$ is uniformly bounded for $\varepsilon\in(0,\varepsilon_{\Omega,j})$. Namely, we have the following proposition. 

\begin{proposition}\label{vjee} There is a constant $C>0$ such that $\sqrt\varepsilon\norm{v_{j,\varepsilon}}_\varepsilon\le C$ for all $\varepsilon\in(0,\varepsilon_{\Omega,j})$.
\end{proposition}
\begin{proof} 
We have 
\begin{equation}\label{vjee.eq1}
\norm{v_{j,\varepsilon}}_\varepsilon=\int_{\omega_\varepsilon}\rho_\varepsilon\,v_{j,\varepsilon}^2dx+\int_{\omega_\varepsilon}|\nabla v_{j,\varepsilon}|^2dx\,.
\end{equation}
Since $\rho_\varepsilon = \varepsilon+\frac{1}{\varepsilon}\tilde\rho_\varepsilon$ on $\omega_\varepsilon$ we have
\[
\int_{\omega_\varepsilon}\rho_\varepsilon\,v_{j,\varepsilon}^2dx= \left(\varepsilon+\frac{1}{\varepsilon}\tilde\rho(\varepsilon)\right)\norm{v_{j,\varepsilon}}^2_{L^2(\Omega)}\,.
\]
Thus, by Proposition \ref{vjeL2} and by \eqref{asymptotic_rho} we deduce that 
\begin{equation}\label{vjee.eq2}
\int_{\omega_\varepsilon}\rho_\varepsilon\,v_{j,\varepsilon}^2dx\le C\qquad\forall\varepsilon\in(0,\varepsilon_{\Omega,j})
\end{equation}
for some $C>0$. By \eqref{grad2} and by the rule of change of variables in integrals we have
\[
\begin{split}
&\int_{\omega_\varepsilon}|\nabla v_{j,\varepsilon}|^2dx\\
&\qquad=\int_0^{|\partial\Omega|}\int_0^1\left(\frac{1}{\varepsilon^2}(\partial_{\xi}w_j(s,\xi))^2+\frac{(\partial_s w_j(s,\xi))^2}{(1-\varepsilon\xi\kappa(s))^2}\right)\varepsilon(1-\varepsilon\xi\kappa(s))\,d\xi ds\\
&\qquad=\frac{1}{\varepsilon} \int_0^{|\partial\Omega|} \int_0^1(\partial_{\xi}w_j(s,\xi))^2 (1-\varepsilon\xi\kappa(s))\,d\xi ds +\varepsilon\int_0^{|\partial\Omega|}\int_0^1 \frac{(\partial_s w_j(s,\xi))^2}{1-\varepsilon\xi\kappa(s)}\,d\xi ds.
\end{split}
\]
From \eqref{w0} we observe that
\begin{equation}\label{regu0}
|\partial_{\xi}w_j(s,\xi)|=\frac{M\mu_j}{|\partial\Omega|}(1-\xi)|u_j\circ\gamma(s)|
\end{equation}
and
\begin{equation}\label{regu1}
|\partial_{s}w_j(s,\xi)|=\frac{M\mu_j}{2|\partial\Omega|}(\xi-1)^2|\partial_s (u_j\circ\gamma)(s)|.
\end{equation}
Since $\Omega$ is assumed to be of class $C^3$, a classical elliptic regularity argument shows that $u_j\in C^2(\overline\Omega)$ (see e.g., Agmon {\it et al.}~\cite{agmon1}). In addition, by the regularity of $\Omega$, we have that $\gamma$ is of class $C^3$ from $[0,|\partial\Omega|)$ to $\mathbb R^2$. Thus, from \eqref{regu0} and \eqref{regu1} it follows that $|\partial_{\xi}w_j(s,\xi)|$, $|\partial_{s}w_j(s,\xi)|\leq C\|u_j\|_{C^1(\overline\Omega)}$. Then by condition \eqref{positiveinf} we verify that 
\begin{equation}\label{vjee.eq3}
\int_{\omega_\varepsilon}|\nabla v_{j,\varepsilon}|^2dx\le C\frac{1}{\varepsilon}\qquad\forall\varepsilon\in(0,\varepsilon_{\Omega,j})
\end{equation}
Now, by \eqref{vjee.eq1}, \eqref{vjee.eq2}, and \eqref{vjee.eq3} we deduce the validity of the proposition.
\end{proof}

We now consider the operator $\mathcal A_{\varepsilon}$ introduced in Section \ref{sec:2}. We recall that $\mathcal A_{\varepsilon}$ is a  compact self-adjoint operator  from $\mathcal H_{\varepsilon}(\Omega)$ to itself. In addition, $\lambda_j(\varepsilon)$ is an eigenvalue of \eqref{Neumann} if and only if $\frac{1}{1+\lambda_j(\varepsilon)}$ is an eigenvalue of $\mathcal A_{\varepsilon}$ and Theorem \ref{convergence} implies that 
\[
\lim_{\varepsilon\to 0}\frac{1}{1+\lambda_{j}(\varepsilon)}=\frac{1}{1+\mu_j}\,.
\]
Since $\mu_j$ is a simple eigenvalue of \eqref{Neumann}, we can prove that $\frac{1}{1+\lambda_j(\varepsilon)}$ is also simple for $\varepsilon$ small enough and we have the following Lemma \ref{only}.

\begin{lemma}\label{only} There exist $\delta_j\in(0,\varepsilon_{\Omega,j})$ and $r^*_j>0$ such  that,  for all $\varepsilon\in(0,\delta_j)$ the only  eigenvalue of $\mathcal A_{\varepsilon}$ in the interval 
\[
\left[\frac{1}{1+{\mu_j}}-r^*_j,\frac{1}{1+{\mu_j}}+r^*_j\right]
\]
is $\frac{1}{1+\lambda_{j}(\varepsilon)}$.  
\end{lemma}
\proof
Since $\mu_j$ and $\lambda_j(\varepsilon)$ are simple we have  $\mu_j\neq\mu_{j-1}$, $\mu_j\neq\mu_{j+1}$, $\lambda_j(\varepsilon)\ne\lambda_{j-1}(\varepsilon)$ and $\lambda_j(\varepsilon)\ne\lambda_{j+1}(\varepsilon)$ for all $\varepsilon\in(0,\varepsilon_{\Omega,j})$. Then, by Theorem \ref{convergence} (i) and by a standard continuity argument we can find  $\delta_j\in(0,\varepsilon_{\Omega,j})$ and $r^*_j>0$ such that 
\[
\left|\frac{1}{1+\mu_j}-\frac{1}{1+\lambda_{j-1}(\varepsilon)}\right|>r^*_j\,,\quad\left|\frac{1}{1+\mu_j}-\frac{1}{1+\lambda_{j+1}(\varepsilon)}\right|>r^*_j\,,
\]
and 
\[
\left|\frac{1}{1+\mu_j}-\frac{1}{1+\lambda_{j}(\varepsilon)}\right|\leq r^*_j
\] for all $\varepsilon\in(0,\delta_j)$. 
\qed

\medskip

To prove Proposition \ref{intermediate} we plan to apply Lemma \ref{lemma_fondamentale} to $\mathcal A_{\varepsilon}$ with $H=\mathcal H_{\varepsilon}(\Omega)$, $\eta=\frac{1}{1+\mu_j}$, $u=\frac{u_j+\varepsilon v_{j,\varepsilon}}{\|u_j+\varepsilon v_{j,\varepsilon}\|_{\varepsilon}}$, and $r=C\varepsilon <r^*_j$, where $C>0$ is a constant which does not depend on $\varepsilon$. Accordingly, we have to verify that the assumptions of Lemma \ref{lemma_fondamentale} are satisfied.  

As a first step, we prove the following
\begin{lemma}\label{C1}
There exists a constant $C_1>0$ such that 
\begin{equation}\label{condition_1}
\left|\left\langle\mathcal A_{\varepsilon}(u_j+\varepsilon v_{j,\varepsilon})-\frac{1}{1+\mu_j}(u_j+\varepsilon v_{j,\varepsilon}),\varphi\right\rangle_{\varepsilon}\right|\leq C_1\varepsilon \|\varphi\|_{\varepsilon}
\end{equation}
for all $\varphi\in\mathcal H_{\varepsilon}(\Omega)$ and for all $\varepsilon\in(0,\varepsilon_{\Omega,j})$.
\end{lemma}

\proof By \eqref{bilinear_eps} and \eqref{A_eps} we have
\begin{equation}\label{mod1_lem}
\begin{split}
&\left|\left\langle\mathcal A_{\varepsilon}(u_j+\varepsilon v_{j,\varepsilon})-\frac{1}{1+\mu_j}(u_j+\varepsilon v_{j,\varepsilon}),\varphi\right\rangle_{\varepsilon}\right|\\
&\quad=\left|\int_{\Omega}\rho_{\varepsilon} u_j\varphi dx+\int_{\omega_{\varepsilon}}\varepsilon\rho_{\varepsilon}v_{j,\varepsilon}\varphi dx
-\frac{1}{1+\mu_j}\left(\int_{\Omega}\nabla u_j\cdot\nabla\varphi dx+\int_{\Omega}\rho_{\varepsilon}u_j\varphi dx\right.\right.\\
&\qquad\qquad\qquad\qquad\qquad\qquad\qquad\qquad+\left.\left.\int_{\omega_{\varepsilon}}\varepsilon\nabla v_{j,\varepsilon}\cdot\nabla\varphi dx+\int_{\omega_{\varepsilon}}\varepsilon\rho_{\varepsilon}v_{j,\varepsilon} \varphi dx\right)\right|\\
&\quad=\frac{\mu_j}{1+\mu_j}\left|\varepsilon\int_{\Omega}u_j \varphi dx+\int_{\omega_{\varepsilon}}\frac{1}{\varepsilon}\tilde\rho_{\varepsilon}u_j\varphi dx-\frac{M}{|\partial\Omega|}\int_{\partial\Omega}u_j\varphi d\sigma\right.\\
&\qquad\qquad\qquad\qquad\qquad\qquad\qquad\qquad\left.+\varepsilon\int_{\omega_{\varepsilon}}\rho_{\varepsilon}v_{j,\varepsilon}\varphi dx-\frac{\varepsilon}{\mu_j}\int_{\omega_{\varepsilon}}\nabla v_{j,\varepsilon}\cdot\nabla\varphi dx\right|
\end{split}
\end{equation}
(see also \eqref{def_tilde_rho} for the definition of $\tilde\rho_{\varepsilon}$). We observe that  by the rule of change of variables in integrals we have
\begin{equation}\label{J2e_eq1.0}
\begin{split}
\int_{\omega_{\varepsilon}}\frac{1}{\varepsilon}\tilde\rho_{\varepsilon}u_j\varphi dx&=\int_0^{|\partial\Omega|}\int_0^1\tilde\rho_{\varepsilon}(u_j\circ\psi_{\varepsilon}(s,\xi))(\varphi\circ\psi_{\varepsilon}(s,\xi))(1-\varepsilon\xi\kappa(s))d\xi ds\\
&=\int_0^{|\partial\Omega|}\int_0^1\tilde\rho_{\varepsilon}(u_j\circ\psi_{\varepsilon})(\varphi\circ\psi_{\varepsilon})d\xi ds\\
&\quad-\int_0^{|\partial\Omega|}\int_0^1\tilde\rho_{\varepsilon}(u_j\circ\psi_{\varepsilon}(s,\xi))(\varphi\circ\psi_{\varepsilon}(s,\xi))\varepsilon\xi\kappa(s) d\xi ds
\end{split}
\end{equation}
 and
\begin{equation}\label{J5_eq1.0}
\begin{split}
&\frac{\varepsilon}{\mu_j}\int_{\omega_{\varepsilon}}\nabla v_{j,\varepsilon}\cdot\nabla\varphi dx\\
&=\frac{\varepsilon^2}{\mu_j}\int_0^{|\partial\Omega|}\int_0^1\left(\frac{1}{\varepsilon^2}\partial_{\xi} w_{j}(s,\xi)\partial_{\xi}(\varphi\circ\psi_{\varepsilon})(s,\xi)+\partial_s w_{j}(s,\xi)\partial_s(\varphi\circ\psi_{\varepsilon})(s,\xi)\right.\\
&\quad\left. +\varepsilon\xi\kappa(s)\xi\sum_{j=1}^{+\infty}(j+1)(\varepsilon\xi\kappa(s))^{j-1}\partial_s w_{j}(s,\xi)\partial_s(\varphi\circ\psi_{\varepsilon})(s,\xi)\right)(1-\varepsilon\xi\kappa(s))d\xi ds\\
&=\frac{1}{\mu_j}\int_0^{|\partial\Omega|}\int_0^1\partial_{\xi} w_{j}(s,\xi)\partial_{\xi}(\varphi\circ\psi_{\varepsilon})(s,\xi)d\xi\,ds\\
&\quad +\frac{1}{\mu_j}\int_0^{|\partial\Omega|}\int_0^1\partial_{\xi} w_{j}(s,\xi)\partial_{\xi}(\varphi\circ\psi_{\varepsilon})\varepsilon\xi\kappa(s)d\xi ds\\
&\quad+\frac{\varepsilon^2}{\mu_j}\int_0^{|\partial\Omega|}\int_0^1\bigg(\partial_s w_{j}(s,\xi)\partial_s(\varphi\circ\psi_{\varepsilon})(s,\xi)\\
&\quad+\varepsilon\xi\kappa(s)\xi\sum_{j=1}^{+\infty}(j+1)(\varepsilon\xi\kappa(s))^{j-1}\partial_s w_{j}(s,\xi)\partial_s(\varphi\circ\psi_{\varepsilon})(s,\xi)\bigg)(1-\varepsilon\xi\kappa(s))d\xi ds
\end{split}
\end{equation}
(see also \eqref{grad2}). In addition, by integrating by parts and by \eqref{w0probl} one verifies that
\begin{equation}\label{J5_eq1.1}
\begin{split}
&\frac{1}{\mu_j}\int_0^{|\partial\Omega|}\int_0^1\partial_{\xi} w_{j}(s,\xi)\partial_{\xi}(\varphi\circ\psi_{\varepsilon})(s,\xi)d\xi\,ds\\
&=-\frac{M}{|\partial\Omega|}\int_{\partial\Omega}u_j\varphi d\sigma+\frac{M}{|\partial\Omega|}\int_0^{|\partial\Omega|}\int_0^1(u_j\circ\psi_{\varepsilon}(s,0))(\varphi\circ\psi_{\varepsilon}(s,\xi)) d\xi ds\,.
\end{split}
\end{equation}
Then by \eqref{J2e_eq1.0}, \eqref{J5_eq1.0}, and \eqref{J5_eq1.1} one deduces that the right hand side of the equality in \eqref{mod1_lem} equals
\[
\frac{\mu_j}{1+\mu_j}\left|J_{1,\varepsilon}+J_{2,\varepsilon}+J_{3,\varepsilon}+J_{4,\varepsilon}+J_{5,\varepsilon}+J_{6,\varepsilon}\right|
\]
with
\begin{eqnarray*}
J_{1,\varepsilon}&:=&\varepsilon\int_{\Omega}u_j \varphi dx,\\
J_{2,\varepsilon}&:=&-\int_0^{|\partial\Omega|}\int_0^1\tilde\rho_{\varepsilon}(u_j\circ\psi_{\varepsilon}(s,\xi))(\varphi\circ\psi_{\varepsilon}(s,\xi))\varepsilon\xi\kappa(s) d\xi ds,\\
J_{3,\varepsilon}&:=&\varepsilon\int_{\omega_{\varepsilon}}\rho_{\varepsilon}v_{j,\varepsilon}\varphi dx,\\
J_{4,\varepsilon}&:=&-\frac{1}{\mu_j}\int_0^{|\partial\Omega|}\int_0^1\partial_{\xi} w_{j}(s,\xi)\partial_{\xi}(\varphi\circ\psi_{\varepsilon})(s,\xi)\varepsilon\xi\kappa(s)d\xi ds\\
J_{5,\varepsilon}&:=&-\frac{\varepsilon^2}{\mu_j}\int_0^{|\partial\Omega|}\int_0^1\bigg(\partial_s w_{j}(s,\xi)\partial_s(\varphi\circ\psi_{\varepsilon})(s,\xi)\\
&&-\varepsilon\xi\kappa(s)\xi\sum_{j=1}^{+\infty}(j+1)(\varepsilon\xi\kappa(s))^{j-1}\partial_s w_{j}(s,\xi)\partial_s(\varphi\circ\psi_{\varepsilon})(s,\xi)\bigg)(1-\varepsilon\xi\kappa(s))d\xi ds,\\
J_{6,\varepsilon}&:=&\int_0^{|\partial\Omega|}\int_0^1\tilde\rho_{\varepsilon}(u_j\circ\psi_{\varepsilon})(\varphi\circ\psi_{\varepsilon})d\xi ds-\frac{M}{|\partial\Omega|}\int_0^{|\partial\Omega|}\int_0^1(u_j\circ\gamma)(\varphi\circ\psi_{\varepsilon}) d\xi ds.\\
\end{eqnarray*}
To prove the validity of the lemma we will show that there exists $C>0$
\begin{equation}\label{aim}
|J_{k,\varepsilon}|\le C\varepsilon\|\varphi\|_{\varepsilon}\qquad\forall\varepsilon\in(0,\varepsilon_{\Omega,j})\,,\varphi\in\mathcal H_{\varepsilon}(\Omega)
\end{equation}
for all $k\in\{1,\dots,6\}$. In the sequel we find convenient to adopt the following convention: we will denote by $C$ a positive constant which does not depend on $\varepsilon$ and $\varphi$ and which may be re-defined line by line.

We begin with $J_{1,\varepsilon}$. We observe that there exists $C>0$ such that 
\begin{equation}\label{J1e_eq2}
\|u_j\|_{\varepsilon}\leq C
\end{equation}
for all $\varepsilon\in(0,\varepsilon_{\Omega,j})$. The proof of \eqref{J1e_eq2} can be effected  by noting that
\begin{equation}\label{J1e_eq2.1}
\lim_{\varepsilon\rightarrow 0}\int_{\Omega}\rho_{\varepsilon}u_j^2 dx=\frac{M}{|\partial\Omega|}\int_{\partial\Omega}u_j^2 d\sigma=\frac{M}{|\partial\Omega|}.
\end{equation}
and by a standard continuity argument. Then, by the  H\"older inequality and  by Proposition \ref{L2<eps} we deduce that
\begin{equation*}
{\left|J_{1,\varepsilon}\right|=\varepsilon\int_{\Omega}\left|u_j\varphi\right| dx}\leq \varepsilon \|u_j\|_{L^2(\Omega)}\|\varphi\|_{L^2(\Omega)}\leq C\varepsilon\|u_j\|_{\varepsilon}\|\varphi\|_{\varepsilon}\leq C\varepsilon\|\varphi\|_{\varepsilon}
\end{equation*}
for all $\varepsilon\in(0,\varepsilon_{\Omega,j})$. Accordingly \eqref{aim} holds with $k=1$.

Now we consider $J_{2,\varepsilon}$. We write
\begin{equation*}
J_{2,\varepsilon}=-\int_0^{|\partial\Omega|}\int_0^1\tilde\rho_{\varepsilon}(u_j\circ\psi_{\varepsilon}(s,\xi))(\varphi\circ\psi_{\varepsilon}(s,\xi))\frac{\xi\kappa(s)}{1-\varepsilon\xi\kappa(s)}\varepsilon(1-\varepsilon\xi\kappa(s))d\xi ds\,.
\end{equation*}
Then we observe that by \eqref{positiveinf} there exists a constant $C$ such that 
\begin{equation}\label{J2e_eq2}
\frac{\xi\kappa(s)}{1-\varepsilon\xi\kappa(s)}<C\qquad\forall\varepsilon\in(0,\varepsilon_{\Omega,j})\,,\;(s,\xi)\in[0,|\partial\Omega|)\times(0,1)\,.
\end{equation}
Hence, by the Cauchy-Schwarz inequality and by \eqref{J1e_eq2} we have
\begin{equation*}
\begin{split}
\left|J_{2,\varepsilon}\right|&\leq C\varepsilon\int_{\omega_{\varepsilon}}\frac{\tilde\rho_{\varepsilon}}{\varepsilon}|u_j\varphi|dx\\
&\leq C\varepsilon\int_{\Omega}\rho_{\varepsilon}|u_j\varphi|dx\leq C\varepsilon\norm{u_j}_{\varepsilon}\norm{\varphi}_{\varepsilon}\leq C\varepsilon\norm{\varphi}_{\varepsilon}\quad\forall\varepsilon\in(0,\varepsilon_{\Omega,j})
\end{split}
\end{equation*}
and the validity of \eqref{aim} with $k=2$ is proved.

We now pass to consider $J_{3,\varepsilon}$. By the H\"older inequality we have
\begin{equation*}
\left|J_{3,\varepsilon}\right|=\varepsilon\int_{\omega_{\varepsilon}}\rho_{\varepsilon}^{\frac{1}{2}}v_{j,\varepsilon}\, \rho_{\varepsilon}^{\frac{1}{2}}\varphi dx\le
\varepsilon\left(\int_{\omega_{\varepsilon}}\rho_{\varepsilon}v^2_{j,\varepsilon}\,dx\right)^{\frac{1}{2}} \left(\int_{\omega_{\varepsilon}}\rho_{\varepsilon}\varphi^2\,dx\right)^{\frac{1}{2}}\,.
\end{equation*}
Then \eqref{aim} with $k=3$ follows by \eqref{vjee.eq2}.

For $J_{4,\varepsilon}$ we observe that we can write
\[
J_{4,\varepsilon}=-\frac{1}{\mu_j}\int_0^{|\partial\Omega|}\int_0^1\partial_{\xi} w_{j}(s,\xi)\partial_{\xi}(\varphi\circ\psi_{\varepsilon})(s,\xi)\frac{\xi\kappa(s)}{1-\varepsilon\xi\kappa(s)}\varepsilon(1-\varepsilon\xi\kappa(s))d\xi ds\,.
\]
Then by \eqref{J2e_eq2}, by the rule of change of variables in integrals, and by the H\"older inequality we have 
\begin{equation*}
\left|J_{4,\varepsilon}\right|\le C\varepsilon\|\nabla\varphi\|_{L^2(\Omega)}\qquad\forall\varepsilon\in(0,\varepsilon_{\Omega,j})\,.
\end{equation*}
Thus \eqref{aim} with $k=4$ follows by the definition of $\|\cdot\|_\varepsilon$ (cf.~\eqref{bilinear_eps}). 

Similarly, by the rule of change of variables in integrals, and by the H\"older inequality one deduces that 
\begin{equation*}
\left|J_{5,\varepsilon}\right|\le C\varepsilon\|\nabla\varphi\|_{L^2(\Omega)}\qquad\forall\varepsilon\in(0,\varepsilon_{\Omega,j})
\end{equation*}
and \eqref{aim} with $k=5$ follows by the definition of $\|\cdot\|_\varepsilon$.

Finally we consider $J_{6,\varepsilon}$. By a straightforward computation one verifies that 
\begin{equation}\label{J7J8}
J_{6,\varepsilon}=J_{7,\varepsilon}+J_{8,\varepsilon}
\end{equation}
with
\[
\begin{split}
J_{7,\varepsilon}&:=\int_0^{|\partial\Omega|}\int_0^1 \left(\tilde\rho_{\varepsilon}-\frac{M}{|\partial\Omega|}\right)(u_j\circ\gamma)(\varphi\circ\psi_{\varepsilon})d\xi ds\,,\\
J_{8,\varepsilon}&:=\int_0^{|\partial\Omega|}\int_0^1\tilde\rho_{\varepsilon}\bigl(u_j\circ\psi_{\varepsilon}(s,\xi))-(u_j\circ\psi_{\varepsilon}(s,0))\bigr)(\varphi\circ\psi_{\varepsilon}(s,\xi)) d\xi ds\,.
\end{split}
\]
We first study $J_{7,\varepsilon}$.  By \eqref{asymptotic_rho} it follows that 
\begin{equation}\label{J7}
\begin{split}
&|J_{7,\varepsilon}|=\left|\int_0^{|\partial\Omega|}\int_0^1 \left(\frac{\frac{1}{2}K M-|\Omega||\partial\Omega|}{|\partial\Omega|^2}\varepsilon+\varepsilon^2\tilde{R}(\varepsilon)\right)(u_j\circ\gamma)(\varphi\circ\psi_{\varepsilon}) d\xi ds\right|\\
&\qquad\leq C\varepsilon\int_0^{|\partial\Omega|}\int_{0}^1|(u_j\circ\gamma)(\varphi\circ\psi_{\varepsilon})|d\xi ds\\
\end{split}
\end{equation}
Hence, by the H\"older inequality, by the rule of change of variables in integrals, by condition \eqref{ucondition}, and by \eqref{positiveinf} we have
\[
\begin{split}
&|J_{7,\varepsilon}|\leq C\varepsilon\left(\int_{\partial\Omega}u_j^2d\sigma\right)^{\frac{1}{2}}\left(\int_0^{|\partial\Omega|}\int_0^1(\varphi\circ\psi_{\varepsilon}(s,\xi))^2\frac{\varepsilon(1-\varepsilon\xi\kappa(s))}{\varepsilon(1-\varepsilon\xi\kappa(s))}d\xi ds\right)^{\frac{1}{2}}\\
&\qquad\leq C\varepsilon \left(\int_{\omega_{\varepsilon}}\frac{1}{\varepsilon}\varphi^2dx\right)^{\frac{1}{2}}\leq C\varepsilon \norm{\varphi}_{\varepsilon}\,.
\end{split}
\]
We now turn to $J_{8,\varepsilon}$. Since $\Omega$ is assumed to be of class $C^{3}$ and $u_j$ is a solution of \eqref{Steklov}, a classical elliptic regularity argument shows that $u_j\in C^{2}(\overline\Omega)$ (see e.g., \cite{agmon1}). In addition,  by the regularity of $\Omega$ we also have that $\psi_{\varepsilon}$ is of class $C^2$ from $[0,|\partial\Omega|)\times(0,1)$ to $\mathbb{R}^2$.  Thus $u_j\circ\psi_{\varepsilon}$ is of class $C^2$ from $[0,|\partial\Omega|)\times(0,1)$ to $\mathbb{R}$ and we can prove that for all $\varepsilon\in(0,\varepsilon_{\Omega,j})$ and $(s,\xi)\in [0,|\partial\Omega|)\times(0,1)$ there exists $\xi^*$ such that 
\[
(u_j\circ\psi_{\varepsilon}(s,\xi))-(u_j\circ\psi_{\varepsilon}(s,0))=\xi\partial_{\xi}(u_j\circ\psi_{\varepsilon})(s,\xi^*).
\]
 Then, by taking $t^*:=\varepsilon\xi^*$ we have
\begin{multline}\label{J801}
J_{8,\varepsilon}=\tilde\rho_{\varepsilon}\int_0^{|\partial\Omega|}\int_0^1\xi\partial_{\xi}(u_j\circ\psi_{\varepsilon})(s,\xi^*)(\varphi\circ\psi_{\varepsilon}(s,\xi))d\xi ds\\
=\tilde\rho_{\varepsilon}\int_0^{|\partial\Omega|}\int_0^{\varepsilon} t\partial_t(u_j\circ\psi)(s,t^*)(\varphi\circ\psi(s,t))\frac{dt}{\varepsilon}ds\,.
\end{multline}
Hence, by the H\"older inequality we deduce that

\begin{multline}\label{J802}
|J_{8,\varepsilon}|\leq C\|u_j\|_{C^{1}(\overline\Omega)}\int_0^{|\partial\Omega|}\int_0^{\varepsilon}\frac{t}{\varepsilon^{\frac{1}{2}}}\frac{|\varphi\circ\psi|}{\varepsilon^{\frac{1}{2}}}dt ds\\
\leq C\|u_j\|_{C^{1}(\overline\Omega)}\left(\int_0^{|\partial\Omega|}\int_0^{\varepsilon}\frac{t^2}{\varepsilon} dt ds\right)^{\frac{1}{2}}\left(\int_0^{|\partial\Omega|}\int_0^{\varepsilon}\frac{(\varphi\circ\psi)^2}{\varepsilon} dt ds\right)^{\frac{1}{2}}\\
= C\|u_j\|_{C^{1}(\overline\Omega)}\frac{|\partial\Omega|^{\frac{1}{2}}}{\sqrt 3}\varepsilon\left(\int_0^{|\partial\Omega|}\int_0^{\varepsilon}\frac{(\varphi\circ\psi)^2}{\varepsilon} dt ds\right)^{\frac{1}{2}}.
\end{multline}

We now observe that we have $|D\psi(s,t)|=1-t\kappa(s)$ for all $(s,t)\in [0,|\partial\Omega|)\times(0,\varepsilon)$ and 
\[
\inf_{(s,t)\in [0,|\partial\Omega|)\times(0,\varepsilon)}(1-t\kappa(s))>0
\]
for all $\varepsilon\in(0,\varepsilon_{\Omega,j})$ (cf.~\eqref{positiveinf}). Thus, by the rule of change of variables in integrals we compute
\begin{equation}\label{J8}
|J_{8,\varepsilon}|\le C\varepsilon\left(\int_0^{|\partial\Omega|}\int_0^{\varepsilon}\frac{(\varphi\circ\psi)^2}{\varepsilon} \frac{1-t\kappa(s)}{1-t\kappa(s)}dt ds\right)^{\frac{1}{2}}\le C\varepsilon\|\phi\|_\varepsilon\,.
\end{equation}
Finally, by \eqref{J7J8}, \eqref{J7}, and \eqref{J8} one deduces that \eqref{aim} holds also for $k=6$. Our proof is now complete.\qed

\medskip

Our next step is to verify that $\norm{u_j+\varepsilon v_{j,\varepsilon}}^2_{\varepsilon}-\frac{M}{|\partial\Omega|}(1+\mu_j)$ is in $O(\varepsilon)$ for $\varepsilon\to 0$. To do so, we prove the following lemma.
\begin{lemma}\label{ultimolemmastep1}
There exists a constant $C>0$ such that
 \begin{equation*}
\left|\|u_j+\varepsilon v_{j,\varepsilon}\|^2_{\varepsilon}-\frac{M}{|\partial\Omega|}(1+\mu_j)\right|\le C\varepsilon\qquad\forall\varepsilon\in(0,\varepsilon_{\Omega,j})\,.
\end{equation*}
\end{lemma}
\proof
A straightforward computation shows that 
\begin{equation*}
\begin{split}
&\|u_j+\varepsilon v_{j,\varepsilon}\|^2_{\varepsilon}-\frac{M}{|\partial\Omega|}(1+\mu_j)\\
&\qquad=\langle u_j+\varepsilon v_{j,\varepsilon}\,,\, u_j+\varepsilon v_{j,\varepsilon}\rangle_\varepsilon-\frac{M}{|\partial\Omega|}(1+\mu_j)=\sum_{k=1}^5L_{k,\varepsilon},
\end{split}
\end{equation*}
with
\small
\begin{eqnarray*}
L_{1,\varepsilon}&:=&\int_{\Omega}\rho_{\varepsilon}u_j^2 dx-\frac{M}{|\partial\Omega|},\\
L_{2,\varepsilon}&:=&\int_{\Omega}|\nabla u_j|^2 dx-\frac{M\mu_j}{|\partial\Omega|},\\
L_{3,\varepsilon}&:=&2\varepsilon\int_{\omega_{\varepsilon}}\rho_{\varepsilon}u_jv_{j,\varepsilon} dx,\\
L_{4,\varepsilon}&:=&2\varepsilon\int_{\omega_{\varepsilon}}\nabla u_j\cdot\nabla v_{j,\varepsilon}dx,\\
L_{5,\varepsilon}&:=&\varepsilon^2\int_{\omega_{\varepsilon}}\rho_{\varepsilon}v_{j,\varepsilon}^2 dx+\varepsilon^2\int_{\omega_{\varepsilon}}|\nabla v_{j,\varepsilon}|^2dx\,.
\end{eqnarray*}
\normalsize
To prove the validity of the lemma we will show that there exists $C>0$ such that 
\begin{equation}\label{LlessCe}
|L_{k,\varepsilon}|\leq C\varepsilon\qquad\forall \varepsilon\in(0,\varepsilon_{\Omega,j})\,,
\end{equation} 
for all $k\in\{1,\dots,5\}$. In the sequel we will denote by $C$ a positive constant which does not depend on $\varepsilon$ and which may be re-defined line by line.

We begin with $L_{1,\varepsilon}$. We observe that by condition \eqref{ucondition} we have
\[
L_{1,\varepsilon}=\frac{1}{\varepsilon}\int_{\omega_{\varepsilon}}\tilde\rho_{\varepsilon}u_j^2 dx-\frac{M}{|\partial\Omega|}\int_{\partial\Omega}u_j^2 d\sigma+\varepsilon\int_{\Omega}u_j^2 dx\,.
\]
Hence, by \eqref{asymptotic_rho} we deduce that
\begin{equation}\label{L1.1}
\begin{split}
L_{1,\varepsilon}&=\frac{M}{|\partial\Omega|}\left(\frac{1}{\varepsilon}\int_{\omega_{\varepsilon}}u_j^2 dx-\int_{\partial\Omega}u_j^2 d\sigma\right)\\
&\quad +\left(\frac{\frac{1}{2}KM-|\Omega||\partial\Omega|}{|\partial\Omega|^2}+\varepsilon \tilde{R}(\varepsilon)\right)\int_{\omega_{\varepsilon}}u_j^2dx+\varepsilon\int_{\Omega}u_j^2 dx.
\end{split}
\end{equation}
Since $\Omega$ is assumed to be of class $C^{3}$ and $u_j$ is a solution of \eqref{Steklov}, a classical elliptic regularity argument shows that $u_j\in C^{2}(\overline\Omega)$ (see e.g., \cite{agmon1}). Then one verifies that 
\begin{equation}\label{L1.2}
\int_{\omega_{\varepsilon}}u_j^2dx\le C|\partial\Omega|\,\|u_j\|^2_{C(\overline\Omega)}\,\varepsilon\qquad\text{and }\qquad\varepsilon\int_{\Omega}u_j^2 dx\le |\Omega|\, \|u_j\|^2_{C(\overline\Omega)}\,\varepsilon\,. 
\end{equation}
In addition,  the map which takes $(s,t)\in[0,|\partial\Omega|)\times(0,\varepsilon)$ to $\tilde u_j(s,t):=(u_j\circ\psi(s,t))^2(1-t\kappa(s))$ is of class $C^2$. It follows that
\begin{equation}\label{L1.3}
\begin{split}
&\left|\frac{1}{\varepsilon}\int_{\omega_{\varepsilon}}u_j^2 dx-\int_{\partial\Omega}u_j^2 d\sigma\right|\\
&\qquad\le\int_0^{|\partial\Omega|}\frac{1}{\varepsilon}\int_0^{\varepsilon}\left|(u_j\circ\psi(s,t))^2(1-t\kappa(s))-(u_j\circ\psi(s,0))^2\right|dtds\\
&\qquad\leq\int_0^{|\partial\Omega|}\frac{1}{\varepsilon}\left(\int_0^{\varepsilon}\norm{\tilde u_j}_{C^1([0,|\partial\Omega|]\times[0,\varepsilon])}\,tdt\right)ds\leq C\varepsilon.
\end{split}
\end{equation}
Then the validity of \eqref{LlessCe} with $k=1$ follows by \eqref{L1.1}, \eqref{L1.2}, and \eqref{L1.3}. 

We now consider $L_{2,\varepsilon}$. Since $u_j$ is an eigenfunction of \eqref{Steklov} a standard argument based on the divergence theorem shows that 
\[
\int_{\Omega}|\nabla u_j|^2 dx=\frac{M\mu_j}{|\partial\Omega|}\int_{\partial\Omega}u_j^2 d\sigma\,.
\]
Then, by condition \eqref{ucondition} we have
\begin{equation*}
L_{2,\varepsilon}=0,
\end{equation*}
which readily implies that \eqref{LlessCe} holds with $k=2$.

To prove \eqref{LlessCe}  for $k=3$ we observe that $\rho_\varepsilon = \varepsilon+\frac{1}{\varepsilon}\tilde\rho_\varepsilon$ on $\omega_\varepsilon$. Thus by a computation based on rule of change of variables in integrals  we have 
\[
\begin{split}
L_{3,\varepsilon}&=2\varepsilon \left(\varepsilon+\frac{1}{\varepsilon}\tilde\rho(\varepsilon)\right)\int_{\omega_{\varepsilon}}u_jv_{j,\varepsilon} dx\\
&= 2\varepsilon^2 \left(\varepsilon+\frac{1}{\varepsilon}\tilde\rho(\varepsilon)\right)\int_0^{|\partial\Omega|}\int_0^1 u_j\circ\psi_\varepsilon(s,\xi)w_j(s,\xi)(1-\varepsilon\xi\kappa(s))\,d\xi ds\,.
\end{split}
\]
Hence,
\[
|L_{3,\varepsilon}|\le 2\varepsilon^2 \left|\varepsilon+\frac{1}{\varepsilon}\tilde\rho(\varepsilon)\right|\left(1+\varepsilon_{\Omega,j}\sup_{s\in[0,|\partial\Omega|)}|\kappa(s)|\right)\norm{u_j}_{L^\infty(\Omega)}\int_0^{|\partial\Omega|}\int_0^1w_j d\xi ds
\]
and the validity of \eqref{LlessCe} with $k=3$ follows by \eqref{asymptotic_rho}.

We now consider the case when  $k=4$. By \eqref{grad2} and by the rule of change of variables in integrals we have
\[
\begin{split}
&L_{4,\varepsilon}\\
&=-2\varepsilon\int_0^{|\partial\Omega|}\int_0^1\left(\frac{1}{\varepsilon^2}\partial_{\xi}(u_j\circ\psi_{\varepsilon}(s,\xi))\partial_{\xi}w_j(s,\xi)+\frac{\partial_s (u\circ\psi_{\varepsilon}(s,\xi))\partial_s w_j(s,\xi)}{(1-\varepsilon\xi\kappa(s))^2}\right)\\
&\quad\qquad\qquad\qquad\qquad\qquad\qquad\qquad\qquad\qquad\qquad\qquad\qquad\times\varepsilon(1-\varepsilon\xi\kappa(s))\,d\xi ds
\end{split}
\]
Now, by equality $\psi_\varepsilon(s,\xi)=\gamma(s)+\epsilon\xi\nu(\gamma(s))$ and by membership of $u_j$ in $C^{2}(\overline\Omega)$, we verify that
\[
\left|\partial_{\xi}(u_j\circ\psi_{\varepsilon}(s,\xi))\right|=\varepsilon\left|\nu(\gamma(s))\cdot\nabla u_j(\psi_\varepsilon(s,\xi))\right|\le\varepsilon \| u_j \|_{C^1(\overline\Omega)}
\]
for all $\varepsilon\in(0,\varepsilon_{\Omega,j})$ and for all $(s,\xi)\in[0,|\partial\Omega|)\times(0,1)$. Hence, by \eqref{positiveinf} and by a straightforward computation, we deduce that \eqref{LlessCe} holds with $k=4$. 

Finally, the validity of \eqref{LlessCe} for  $k=5$ is a consequence of Proposition \ref{vjee} and of equality $L_{5,\varepsilon}=\varepsilon^2\norm{v_{j,\varepsilon}}^2_\varepsilon$.

\endproof

\medskip

We are now ready to prove Proposition \ref{intermediate} by Lemma \ref{lemma_fondamentale}. 

\medskip

\noindent{\em Proof of Proposition \ref{intermediate}.}
We first prove \eqref{intermediate.eq1}. By Lemma \ref{ultimolemmastep1} there exists $\varepsilon^*_j\in(0,\varepsilon_{\Omega,j})$ such that 
\[
\|u_j+\varepsilon v_{j,\varepsilon}\|_{\varepsilon}>\frac{1}{2}\sqrt{\frac{M}{|\partial\Omega|}}(1+\mu_j)^{\frac{1}{2}}\qquad\forall\varepsilon\in(0,\varepsilon^*_j)\,.
\]  
Hence, by multiplying both sides of \eqref{condition_1} by $\norm{u_j+\varepsilon v_{j,\varepsilon}}_{\varepsilon}^{-1}$ we deduce that
\begin{equation}\label{condition_11}
\left|\left\langle\mathcal A_{\varepsilon}\left(\frac{u_j+\varepsilon v_{j,\varepsilon}}{\|u_j+\varepsilon v_{j,\varepsilon}\|_{\varepsilon}}\right)-\frac{1}{1+\mu_j}\left(\frac{u_j+\varepsilon v_{j,\varepsilon}}{\|u_j+\varepsilon v_{j,\varepsilon}\|_{\varepsilon}}\right),\varphi\right\rangle_{\varepsilon}\right|\leq C_2\,\varepsilon\, \|\varphi\|_{\varepsilon}
\end{equation}
for all $\varphi\in H^1(\Omega)$ and $\varepsilon\in(0,\varepsilon^*_j)$, with $C_2:=2\sqrt{\frac{|\partial\Omega|}{M}}(1+\mu_j)^{-\frac{1}{2}}C_1$. By taking $\varphi=\mathcal A_{\varepsilon}\left(\frac{u_j+\varepsilon v_{j,\varepsilon}}{\|u_j+\varepsilon v_{j,\varepsilon}\|_{\varepsilon}}\right)-\frac{1}{1+\mu_j}\left(\frac{u_j+\varepsilon v_{j,\varepsilon}}{\|u_j+\varepsilon v_{j,\varepsilon}\|_{\varepsilon}}\right)$ in \eqref{condition_11}, we obtain
\begin{equation*}
\left\|\mathcal{A}_{\varepsilon}\left(\frac{u_j+\varepsilon v_{j,\varepsilon}}{\|u_j+\varepsilon v_{j,\varepsilon}\|_{\varepsilon}}\right)-\frac{1}{1+\mu_j}\left(\frac{u_j+\varepsilon v_{j,\varepsilon}}{\|u_j+\varepsilon v_{j,\varepsilon}\|_{\varepsilon}}\right)\right\|_{\varepsilon}\leq C_2\, \varepsilon\qquad\forall\varepsilon\in(0,\varepsilon^*_j)\,.
\end{equation*}
As a consequence, one can verify that the assumptions of Lemma \ref{lemma_fondamentale}  hold with $A=\mathcal{A}_\varepsilon$, $H=\mathcal H_{\varepsilon}(\Omega)$, $\eta=\frac{1}{1+\mu_j}$, $u=\frac{u_j+\varepsilon v_{j,\varepsilon}}{\|u_j+\varepsilon v_{j,\varepsilon}\|_{\varepsilon}}$, and $r=C_2\,\varepsilon$ with $\varepsilon\in(0,\varepsilon^*_j)$ (see also Proposition \ref{Ae}). Accordingly, for all $\varepsilon\in(0,\varepsilon^*_j)$ there exists  an eigenvalue $\eta^*_\varepsilon$ of $\mathcal A_{\varepsilon}$ such that
\begin{equation}\label{quasi_auto_1}
\left|\frac{1}{1+\mu_j}-\eta^*_\varepsilon\right|\leq C_2\varepsilon.
\end{equation}
Now we take $\varepsilon_{\Omega,j}^\#:=\min\{\varepsilon^*_j\,,\,\delta_j\,,\,C_2^{-1}r^*_j\}$ with $\delta_j$ and $r^*_j$ as in Lemma \ref{only}. By \eqref{quasi_auto_1} and Lemma \ref{only}, the eigenvalue $\eta^*_\varepsilon$ has to coincide with $\frac{1}{1+\lambda_j(\varepsilon)}$ for all $\varepsilon\in(0,\varepsilon_{\Omega,j}^\#)$. It follows  that
\begin{equation*}
\left| \mu_j- \lambda_j(\varepsilon)\right|\leq C_2\left|(1+\mu_j)(1+\lambda_j(\varepsilon))\right|\varepsilon\qquad\forall\varepsilon\in(0,\varepsilon_{\Omega,j}^\#). 
\end{equation*}
Then the validity of \eqref{intermediate.eq1}   follows by Theorem \ref{convergence} (i) and by a straightforward computation.

We now consider \eqref{intermediate.eq2}.  By Lemma \ref{lemma_fondamentale} with $r^*=r^*_j$ it follows that for all $\varepsilon\in(0,\varepsilon_{\Omega,j}^\#)$ there exists a function $u^*_\varepsilon\in\mathcal H_{\varepsilon}(\Omega)$ with $\|u^*_\varepsilon\|_{\varepsilon}=1$ which belongs to the space generated by all the eigenfunctions of $\mathcal{A}_\varepsilon$ associated with the eigenvalues contained in the segment $\left[\frac{1}{1+\mu_j}-r^*_j,\frac{1}{1+\mu_j}+r^*_j\right]$ and such that 
\begin{equation}\label{u*ineq}
\left\|u^*_\varepsilon-\frac{u_j+\varepsilon v_{j,\varepsilon}}{\|u_j+\varepsilon v_{j,\varepsilon}\|_{\varepsilon}}\right\|_{\varepsilon}\leq\frac{2C_2}{r^*_j}\varepsilon.
\end{equation}
Since $\varepsilon\in(0,\varepsilon^\#_{\Omega,j})$, Lemma \ref{only} implies that $\frac{1}{1+\lambda_j(\varepsilon)}$ is the only eigenvalue of $\mathcal A_{\varepsilon}$ which belongs to the segment $\left[\frac{1}{1+\mu_j}-r^*_j,\frac{1}{1+\mu_j}+r^*_j\right]$. In addition $\lambda_j(\varepsilon)$ is simple for $\varepsilon<\varepsilon_{\Omega,j}^\#$ (because $\varepsilon_{\Omega,j}^\#\le \varepsilon_{\Omega,j})$. It follows that $u^*_\varepsilon$ coincides with the only eigenfunction with norm one corresponding to $\lambda_j(\varepsilon)$, namely
$u^*_\varepsilon=\frac{u_{j,\varepsilon}}{\|u_{j,\varepsilon}\|_{\varepsilon}}$. 
Thus by \eqref{u*ineq} we have
\begin{equation}\label{intermediate.eq3}
\left\|\frac{u_{j,\varepsilon}}{\|u_{j,\varepsilon}\|_{\varepsilon}}- \frac{u_j+\varepsilon v_{j,\varepsilon}}{\|u_j+\varepsilon v_{j,\varepsilon}\|_{\varepsilon}}\right\|_{\varepsilon}\leq \frac{2C_2}{r^*_j}\varepsilon\qquad\forall\varepsilon\in(0,\varepsilon^\#_{\Omega,j}).
\end{equation}
We plan to prove that \eqref{intermediate.eq3} implies that
\begin{equation}\label{ineq_eigenfunctions_step1}
\|u_{j,\varepsilon}-u_j-\varepsilon v_{j,\varepsilon}\|_{L^2(\Omega)}\leq C_3 \varepsilon\qquad\forall\varepsilon\in(0,\varepsilon^\#_j)
\end{equation}
for some $C_3>0$. Then the validity of \eqref{intermediate.eq2} will follow by Proposition \ref{vjeL2}. To do so, we observe that by \eqref{uecondition} we have
\begin{equation}\label{intermediate.eq4}
\norm{u_{j,\varepsilon}}_\varepsilon=\sqrt\frac{M}{|\partial\Omega|}(1+\lambda_j(\varepsilon))^{\frac{1}{2}}\qquad\forall\varepsilon\in(0,\varepsilon^\#_{\Omega,j}).
\end{equation}
It follows that 

\begin{multline}\label{intermediate.eq45}
\norm{u_{j,\varepsilon}}_\varepsilon-\norm{u_{j}+\varepsilon v_{j,\varepsilon}}_\varepsilon\\
=\left(\sqrt\frac{M}{|\partial\Omega|}(1+\lambda_j(\varepsilon))^{\frac{1}{2}}-\sqrt\frac{M}{|\partial\Omega|}(1+\mu_j)^{\frac{1}{2}}\right)+\left(\sqrt\frac{M}{|\partial\Omega|}(1+\mu_j)^{\frac{1}{2}}-\norm{u_{j}+\varepsilon v_{j,\varepsilon}}_\varepsilon\right)
\end{multline}

Then a computation based on \eqref{intermediate.eq1} and on Lemma \ref{ultimolemmastep1} shows that
\begin{equation}\label{intermediate.eq5}
\norm{u_{j,\varepsilon}}_\varepsilon-\norm{u_{j}+\varepsilon v_{j,\varepsilon}}_\varepsilon< C_4\, \varepsilon\qquad\forall\varepsilon\in(0,\varepsilon^\#_{\Omega,j})
\end{equation}
for some $C_4>0$.
Now we note that

\begin{multline}\label{intermediate.eq6}
\|u_{j,\varepsilon}-u_j-\varepsilon v_{j,\varepsilon}\|_{\varepsilon}=\norm{\|u_{j,\varepsilon}\|_{\varepsilon}\frac{u_{j,\varepsilon}}{\|u_{j,\varepsilon}\|_{\varepsilon}}-\|u_j+\varepsilon v_{j,\varepsilon}\|_{\varepsilon} \frac{u_j+\varepsilon v_{j,\varepsilon}}{\|u_j+\varepsilon v_{j,\varepsilon}\|_{\varepsilon}}}_\varepsilon\\
\le \norm{u_{j,\varepsilon}}_\varepsilon\norm{\frac{u_{j,\varepsilon}}{\|u_{j,\varepsilon}\|_{\varepsilon}}- \frac{u_j+\varepsilon v_{j,\varepsilon}}{\|u_j+\varepsilon v_{j,\varepsilon}\|_{\varepsilon}}}_\varepsilon+\left|\norm{u_{j,\varepsilon}}_\varepsilon-\norm{u_{j}+\varepsilon v_{j,\varepsilon}}_\varepsilon\right|
\end{multline}
Hence, by \eqref{intermediate.eq3}, \eqref{intermediate.eq4}, and \eqref{intermediate.eq5} we deduce that 
\[
\|u_{j,\varepsilon}-u_j-\varepsilon v_{j,\varepsilon}\|_\varepsilon\leq C_5\, \varepsilon\qquad\forall\varepsilon\in(0,\varepsilon^\#_j)
\]
for some $C_5>0$. Now the validity of \eqref{ineq_eigenfunctions_step1} follows by Proposition \ref{L2<eps}.
 \qed


\section{Second Step}\label{sec:5}

In this section we complete the  proof of Theorems \ref{asymptotic_eigenvalues} and \ref{asymptotic_eigenfunctions}. Accordingly, we fix $j\in\mathbb N$ and we take $\mu_j$, $\mu_j^1$, $u_j$, $\varepsilon_{\Omega,j}$,  $\lambda_j(\varepsilon)$, and $u_{j,\varepsilon}$ as in the statements of Theorems \ref{asymptotic_eigenvalues} and \ref{asymptotic_eigenfunctions}. 

We denote by   $u_j^1$ the unique solution in $H^1(\Omega)$ of the boundary value problem
\begin{equation}\label{u1_problem_simple_true}
\begin{cases}
-\Delta u_j^1=\mu_j u_j, & {\rm in}\ \Omega,\\
\partial_{\nu}u_j^1-\frac{M\mu_j}{|\partial\Omega|}u_j^1=\left(\frac{M\mu_j}{2|\partial\Omega|^2}(K-|\partial\Omega|\kappa{\circ\gamma^{(-1)}})-\frac{2M^2\mu_j^2}{3|\partial\Omega|^2}-\frac{|\Omega|\mu_j}{|\partial\Omega|}\right)u_j+\frac{M\mu_j^1}{|\partial\Omega|}u_j, & {\rm on}\ \partial\Omega,
\end{cases}
\end{equation}
which satisfies the additional condition
\begin{equation}\label{uj1condition}
\int_{\partial\Omega}u_j^1 u_j d\sigma=\left(\frac{\mu_j^1}{2\mu_j}+\frac{M\mu_j}{3|\partial\Omega|}\right)\,.
\end{equation}
The existence and uniqueness of  $u_j^1$ is a consequence of Proposition \ref{A1} in the Appendix. Then we introduce the auxiliary function $w_j^1(s,\xi)$ from $[0,|\partial\Omega|)\times [0,1]$ to $\mathbb R$ defined by
\begin{equation}\label{w1}
\begin{split}
w_j^1(s,\xi)&:=-\frac{\kappa(s)M\mu_j}{6|\partial\Omega|}(u_j\circ\psi_{\varepsilon}(s,0))(\xi-1)^3\\
&\quad+\frac{M^2\mu_j^2}{24|\partial\Omega|^2}(u_j\circ\psi_{\varepsilon}(s,0))(\xi^2+2\xi+9)(\xi-1)^2\\
&\quad+\left(\frac{|\Omega|\mu_j}{2|\partial\Omega|}(u_j\circ\psi_{\varepsilon}(s,0))-\frac{M}{2|\partial\Omega|}(\mu_j (u_j^1\circ\psi_{\varepsilon}(s,0))\right.\\
&\qquad\qquad \left.+\mu_j^1 (u_j\circ\psi_{\varepsilon}(s,0)))-\frac{KM\mu_j}{4|\partial\Omega|^2}(u_j\circ\psi_{\varepsilon}(s,0))\right)(\xi-1)^2,
\end{split}
\end{equation}
for all $(s,\xi)\in[0,|\partial\Omega|)\times[0,1]$ (see \eqref{K} for the definition of $K$). A straightforward computation shows that {}
\[
\begin{split}
&-\partial^2_{\xi}w_j^1(s,\xi)\\
&=-\kappa(s)\partial_{\xi}w_j(s,\xi)+\frac{M}{|\partial\Omega|}\bigg(\mu_j (u_j^1\circ\psi_{\varepsilon}(s,0))+\mu_jw_j(s,\xi)+\mu_j^1(u_j\circ\psi_{\varepsilon}(s,0))\\
&\qquad\qquad-\xi\frac{M\mu_j^2}{|\partial\Omega|}(u_j\circ\psi_{\varepsilon}(s,0))-\frac{|\Omega|\mu_j}{M}(u_j\circ\psi_{\varepsilon}(s,0))+\frac{K\mu_j}{2|\partial\Omega|}(u_j\circ\psi_{\varepsilon}(s,0))\bigg) 
\end{split}
\]
for all $(s,\xi)\in[0,\partial\Omega)\times (0,1)$. Moreover, $w_j^1$ satisfies
\begin{equation*}
w_j^1(s,1)=\partial_{\xi}w_j^1(s,1)=0
\end{equation*}
for all $s\in[0,|\partial\Omega|)$. Then, for all $\varepsilon\in(0,\varepsilon_{\Omega,j})$ we denote by $v_{j,\varepsilon}^1\in H^1(\Omega)$ the extension by $0$ of $w_j^1\circ\psi_{\varepsilon}^{(-1)}$ to $\Omega$. We note that by construction $v_{j,\varepsilon}^1\in H^1(\Omega)$. We also observe that the $L^2(\Omega)$ norm of $v_{j,\varepsilon}^1$ is in $O(\sqrt{\varepsilon})$ as $\varepsilon\rightarrow 0$. Indeed, we have the following proposition.

\begin{proposition}\label{vj1eL2} There is a constant $C>0$ such that $\|v_{j,\varepsilon}^1\|_{L^2(\Omega)}\le C\sqrt\varepsilon$ for all $\varepsilon\in(0,\varepsilon_{\Omega,j})$.
\end{proposition}

The proof is similar to that of Proposition \ref{vjeL2} and it is accordingly omitted. We also observe that $\sqrt{\varepsilon}\|v_{j,\varepsilon}^1\|_{\varepsilon}$ is uniformly bounded for $\varepsilon\in(0,\varepsilon_{\Omega,j})$, as it is stated in the following proposition.

\begin{proposition}\label{vj2ee} There is a constant $C>0$ such that $\sqrt\varepsilon\norm{v_{j,\varepsilon}^1}_\varepsilon\le C$ for all $\varepsilon\in(0,\varepsilon_{\Omega,j})$.
\end{proposition}
The proof of Proposition \ref{vj2ee} can be effected by following  the footsteps of the proof of Proposition \ref{vjee} and it is accordingly omitted. 

Possibly choosing smaller values for $r^*_j$ and $\delta_j$, we have the following Lemma \ref{only2}, which is the analogue of Lemma \ref{only}.

\begin{lemma}\label{only2} There exist $\delta_j\in(0,\varepsilon_{\Omega,j})$ and $r^*_j>0$ such  that,  for all $\varepsilon\in(0,\delta_j)$ the only  eigenvalue of $\mathcal A_{\varepsilon}$ in the interval 
\[
\left[\frac{1}{1+{\mu_j+\varepsilon\mu_j^1}}-r^*_j,\frac{1}{1+{\mu_j+\varepsilon\mu_j^1}}+r^*_j\right]
\]
is $\frac{1}{1+\lambda_{j}(\varepsilon)}$.  
\end{lemma}
The proof of Lemma \ref{only2} is similar to that of Lemma \ref{only} and accordingly it is omitted.

We now consider the operator $\mathcal A_{\varepsilon}$ introduced in Section \ref{sec:2}. In order to complete the proof of Theorems \ref{asymptotic_eigenvalues} and \ref{asymptotic_eigenfunctions} we plan to apply Lemma \ref{lemma_fondamentale} to $\mathcal A_{\varepsilon}$ with
\[
 \text{$H=\mathcal H_{\varepsilon}(\Omega)$, $\eta=\frac{1}{1+\mu_j+\varepsilon\mu_j^1}$, $u=\frac{u_j+\varepsilon v_{j,\varepsilon}+\varepsilon u_j^1+\varepsilon^2 v_{j,\varepsilon}^1}{u_j+\varepsilon v_{j,\varepsilon}+\varepsilon u_j^1+\varepsilon^2 v_{j,\varepsilon}^1}$, and $r=C\varepsilon^2<r^*_j$,}
\]
where $C>0$ is a constant which does not depend on $\varepsilon$. As we did in Section \ref{sec:4}, we have to verify that the assumptions of Lemma \ref{lemma_fondamentale} are satisfied. We observe here that, due to Proposition \eqref{vj1eL2}, the $L^2$ the norm of $\varepsilon^2 v^1_{j,\varepsilon}$ is in $o(\varepsilon^2)$ and accordingly the term $\varepsilon^2 v^1_{j,\varepsilon}$ is negligible from the approximation \eqref{expansion_eigenfunctions}. However, since we will deduce \eqref{expansion_eigenfunctions} from a suitable approximation in $\|\cdot\|_\varepsilon$ norm (cf.~inequality \eqref{r*step2} below), we have to take into account also the contribution of $\varepsilon^2 v^1_{j,\varepsilon}$ (see also Proposition \ref{vj2ee}).

We begin with the following lemma.

\begin{lemma}
 There exists a constant $C_6>0$ such that
\small
\begin{equation}\label{condition_2}
\left|\left\langle\mathcal A_{\varepsilon}(u_j+\varepsilon v_{j,\varepsilon}+\varepsilon u_j^1+\varepsilon^2 v_{j,\varepsilon}^1)-\frac{1}{1+\mu_j+\varepsilon\mu_j^1}(u_j+\varepsilon v_{j,\varepsilon}+\varepsilon u_j^1+\varepsilon^2 v_{j,\varepsilon}^1),\varphi\right\rangle_{\varepsilon}\right|\leq C_6\varepsilon^2 \|\varphi\|_{\varepsilon},
\end{equation}
\normalsize
for all $\varphi\in \mathcal H_{\varepsilon}(\Omega)$ and for all $\varepsilon\in(0,\varepsilon_{\Omega,j})$.
\proof
By \eqref{bilinear_eps} and \eqref{A_eps} we have
\begin{equation}\label{eprod}
\begin{split}
&\left|\left\langle\mathcal A_{\varepsilon}(u_j+\varepsilon v_{j,\varepsilon}+\varepsilon u_j^1+\varepsilon^2 v_{j,\varepsilon}^1)-\frac{1}{1+\mu_j+\varepsilon\mu_j^1}(u_j+\varepsilon v_{j,\varepsilon}+\varepsilon u_j^1+\varepsilon^2 v_{j,\varepsilon}^1),\varphi\right\rangle_{\varepsilon}\right|\\
&=\frac{1}{|1+\mu_j+\varepsilon\mu_j^1|}\left|(\mu_j+\varepsilon\mu_j^1)\int_{\Omega}\rho_{\varepsilon}(u_j+\varepsilon v_{j,\varepsilon}+\varepsilon u_j^1+\varepsilon^2 v_{j,\varepsilon}^1)\varphi dx\right.\\
&\qquad\qquad\qquad\qquad\qquad\qquad\qquad\qquad\left.-\int_{\Omega}\nabla (u_j+\varepsilon v_{j,\varepsilon}+\varepsilon u_j^1+\varepsilon^2 v_{j,\varepsilon}^1)\cdot\nabla\varphi dx\right|.
\end{split}
\end{equation}
We consider the summands appearing in the absolute value on the right-hand side of equality \eqref{eprod} separately and we re-organize them in a more suitable way. We start with the terms involving $u_j$ and $u_j^1$. We have
\begin{equation}\label{xy0}
\mu_j\int_{\Omega}\rho_{\varepsilon}u_j\varphi dx=\varepsilon\mu_j\int_{\Omega}u_j\varphi dx+\mu_j\int_{\omega_{\varepsilon}}\frac{\tilde\rho_{\varepsilon}}{\varepsilon}u_j\varphi dx.
\end{equation}
By using \eqref{asymptotic_rho} we observe that
\begin{multline}\label{new1}
\mu_j\int_{\omega_{\varepsilon}}\frac{\tilde\rho_{\varepsilon}}{\varepsilon}u_j\varphi dx=\mu_j\int_{\omega_{\varepsilon}}\frac{M}{\varepsilon|\partial\Omega|}u_j\varphi dx+\mu_j\int_{\omega_{\varepsilon}}\frac{K M}{2 |\partial\Omega|^2}u_j\varphi dx\\
-\mu_j\int_{\omega_{\varepsilon}}\frac{|\Omega|}{|\partial\Omega|}u_j\varphi dx+\mu_j\,\varepsilon\,\tilde{R}(\varepsilon)\int_{\omega_{\varepsilon}}u_j\varphi dx.
\end{multline}
By the rule of change of variables in integrals we have for the first term in the right-hand side of \eqref{new1}
\begin{multline}\label{new2}
\mu_j\int_{\omega_{\varepsilon}}\frac{M}{\varepsilon|\partial\Omega|}u_j\varphi dx
=\mu_j\int_0^{|\partial\Omega|}\int_0^1\frac{M}{|\partial\Omega|}(u_j\circ\psi_{\varepsilon})(\varphi\circ\psi_{\varepsilon}) d\xi ds\\
-\mu_j\varepsilon\int_0^{|\partial\Omega|}\int_0^1\frac{M}{|\partial\Omega|}\xi\kappa(s)(u_j\circ\psi_{\varepsilon}(s,\xi))(\varphi\circ\psi_{\varepsilon}(s,\xi)) d\xi ds,
\end{multline}
while for the second term in the right-hand side of \eqref{new1} we have
\begin{multline}\label{new3}
\mu_j\int_{\omega_{\varepsilon}}\frac{K M}{2|\partial\Omega|^2}u_j\varphi dx
=\mu_j\varepsilon\int_0^{|\partial\Omega|}\int_0^1\frac{K M}{2|\partial\Omega|^2}(u_j\circ\psi_{\varepsilon})(\varphi\circ\psi_{\varepsilon}) d\xi ds\\
-\mu_j\varepsilon^2\int_0^{|\partial\Omega|}\int_0^1\frac{K M}{2|\partial\Omega|^2}\xi\kappa(s)(u_j\circ\psi_{\varepsilon}(s,\xi))(\varphi\circ\psi_{\varepsilon}(s,\xi)) d\xi ds.
\end{multline}
For the third term in the right-hand side of \eqref{new1} we have
\begin{multline}\label{new4}
-\mu_j\int_{\omega_{\varepsilon}}\frac{|\Omega|}{|\partial\Omega|}u_j\varphi dx
=-\mu_j\varepsilon\int_0^{|\partial\Omega|}\int_0^1\frac{|\Omega|}{|\partial\Omega|}(u_j\circ\psi_{\varepsilon})(\varphi\circ\psi_{\varepsilon}) d\xi ds\\
-\mu_j\varepsilon^2\int_0^{|\partial\Omega|}\int_0^1\frac{|\Omega|}{|\partial\Omega|}\xi\kappa(s)(u_j\circ\psi_{\varepsilon}(s,\xi))(\varphi\circ\psi_{\varepsilon}(s,\xi)) d\xi ds.
\end{multline}
We set
\begin{equation}\label{rem1}
\begin{split}
R_1(\varepsilon)&:=\mu_j \,\varepsilon\,\tilde{R}(\varepsilon)\int_{\omega_{\varepsilon}}u_j\varphi dx\\
&\quad-\mu_j\varepsilon^2\int_0^{|\partial\Omega|}\int_0^1\frac{K M}{2|\partial\Omega|^2}\xi\kappa(s)(u_j\circ\psi_{\varepsilon}(s,\xi))(\varphi\circ\psi_{\varepsilon}(s,\xi)) d\xi ds\\
&\quad-\mu_j\varepsilon^2\int_0^{|\partial\Omega|}\int_0^1\frac{|\Omega|}{|\partial\Omega|}\xi\kappa(s)(u_j\circ\psi_{\varepsilon}(s,\xi))(\varphi\circ\psi_{\varepsilon}(s,\xi)) d\xi ds.
\end{split}
\end{equation}
Then, by \eqref{xy0}-\eqref{rem1}, we have
\begin{equation}\label{piece1}
\begin{split}
\mu_j\int_{\Omega}\rho_{\varepsilon}u_j\varphi dx&=\varepsilon\mu_j\int_{\Omega}u_j\varphi dx+\mu_j\int_0^{|\partial\Omega|}\int_0^1\frac{M}{|\partial\Omega|}(u_j\circ\psi_{\varepsilon})(\varphi\circ\psi_{\varepsilon}) d\xi ds\\
&\quad-\mu_j\varepsilon\int_0^{|\partial\Omega|}\int_0^1\frac{M}{|\partial\Omega|}\xi\kappa(s)(u_j\circ\psi_{\varepsilon}(s,\xi))(\varphi\circ\psi_{\varepsilon}(s,\xi)) d\xi ds\\
&\quad+\mu_j\varepsilon\int_0^{|\partial\Omega|}\int_0^1\frac{K M}{2|\partial\Omega|^2}(u_j\circ\psi_{\varepsilon})(\varphi\circ\psi_{\varepsilon}) d\xi ds\\
&\quad-\mu_j\varepsilon\int_0^{|\partial\Omega|}\int_0^1\frac{|\Omega|}{|\partial\Omega|}(u_j\circ\psi_{\varepsilon})(\varphi\circ\psi_{\varepsilon}) d\xi ds+R_1(\varepsilon).
\end{split}
\end{equation}

In a similar way we observe that
\begin{equation}\label{piece2}
\varepsilon\mu_j\int_{\Omega}\rho_{\varepsilon}u_j^1\varphi dx=\varepsilon\mu_j\int_0^{|\partial\Omega|}\int_0^1\frac{M}{|\partial\Omega|}(u_j^1\circ\psi_{\varepsilon}{})(\varphi\circ\psi_{\varepsilon}{})d\xi ds+R_2(\varepsilon),
\end{equation}
where
\begin{equation}\label{rem2}
\begin{split}
R_2(\varepsilon)&:=\varepsilon^2\mu_j\int_{\Omega}u_j^1\varphi dx+\mu_j\varepsilon\int_{\omega_{\varepsilon}}\left(\frac{KM}{2|\partial\Omega|}-\frac{|\Omega|}{|\partial\Omega|}+\varepsilon\tilde R(\varepsilon)\right)u_j^1\varphi dx\\
&\quad-\varepsilon^2\mu_j\int_0^{|\partial\Omega|}\int_0^1\frac{M}{|\partial\Omega|}(u_j^1\circ\psi_{\varepsilon}(s,\xi))(\varphi\circ\psi_{\varepsilon}(s,\xi))\xi\kappa(s)d\xi ds.
\end{split}
\end{equation}
We also observe that
\begin{equation}\label{piece3}
\varepsilon\mu_j^1\int_{\Omega}\rho_{\varepsilon}u_j\varphi dx=\varepsilon\mu_j^1\int_0^{|\partial\Omega|}\int_0^1\frac{M}{|\partial\Omega|}(u_j\circ\psi_{\varepsilon}(s,\xi))(\varphi\circ\psi_{\varepsilon}(s,\xi))d\xi ds+R_3(\varepsilon),
\end{equation}
where
\begin{equation*}
\begin{split}
R_3(\varepsilon)&:=\varepsilon^2\mu_j^1\int_{\Omega}u_j\varphi dx+\mu_j^1\varepsilon\int_{\omega_{\varepsilon}}\left(\frac{KM}{2|\partial\Omega|}-\frac{|\Omega|}{|\partial\Omega|}+\varepsilon\tilde R(\varepsilon)\right)u_j\varphi dx\\
&\quad-\varepsilon^2\mu_j^1\int_0^{|\partial\Omega|}\int_0^1\frac{M}{|\partial\Omega|}(u_j\circ\psi_{\varepsilon}(s,\xi))(\varphi\circ\psi_{\varepsilon}(s,\xi))\xi\kappa(s)d\xi ds.
\end{split}
\end{equation*}
We also find convenient to set
\begin{equation}\label{rem4}
R_4(\varepsilon):=\varepsilon^2\mu_j^1\int_{\Omega}u_j^1\varphi dx.
\end{equation}

Since $u_j$ is an eigenfunction of \eqref{Steklov} associated with the eigenvalue $\mu_j$, a standard argument based on the divergence theorem shows that
\begin{equation}\label{piece5}
\int_{\Omega}\nabla u_j\cdot\nabla\varphi dx=\int_{\partial\Omega}\frac{M\mu_j}{|\partial\Omega|}u_j\varphi d\sigma.
\end{equation}
Moreover, $u_j^1,\mu_j^1$ are solutions to problem \eqref{u1_problem_simple_true} and therefore
\begin{equation}\label{piece6}
\begin{split}
&\varepsilon\int_{\Omega}\nabla u_j^1\cdot\nabla\varphi dx\\
&\quad=\varepsilon\int_{\Omega}\mu_ju_j\varphi dx\\
&\qquad+\varepsilon\int_{\partial\Omega}\left(\frac{M\mu_j}{2|\partial\Omega|^2}(K-|\partial\Omega|\kappa(s))-\frac{2M^2\mu_j^2}{3|\partial\Omega|^2}+\frac{M\mu_j^1-|\Omega|\mu_j}{|\partial\Omega|}\right)u_j\varphi d\sigma\\
&\qquad+\varepsilon\int_{\partial\Omega}\frac{M\mu_j}{|\partial\Omega|}u_j^1\varphi d\sigma.
\end{split}
\end{equation}

Now we consider the terms involving $v_{j,\varepsilon}$ and $v_{j,\varepsilon}^1$. We have
\begin{equation}\label{sopra}
\varepsilon\mu_j\int_{\Omega}\rho_{\varepsilon}v_{j,\varepsilon}\varphi dx=\varepsilon^2\mu_j\int_{\omega_{\varepsilon}}v_{j,\varepsilon}\varphi dx+\mu_j\int_{\omega_{\varepsilon}}\tilde\rho_{\varepsilon}v_{j,\varepsilon}\varphi dx
\end{equation}
By the rule of change of variables in integrals and by \eqref{asymptotic_rho} we observe that for the second summand in the right-hand side of \eqref{sopra} it holds
\begin{equation*}
\begin{split}
&\mu_j\int_{\omega_{\varepsilon}}\tilde\rho_{\varepsilon}v_{j,\varepsilon}\varphi dx\\
&=\varepsilon\mu_j\int_0^{|\partial\Omega|}\int_0^1\frac{M}{|\partial\Omega|}w_j(\varphi\circ\psi_{\varepsilon})d\xi ds\\
&\quad-\varepsilon^2\mu_j \int_0^{|\partial\Omega|}\int_0^1\frac{M}{|\partial\Omega|}\kappa(s)\xi w_j(s,\xi)(\varphi\circ\psi_{\varepsilon}(s,\xi))d\xi ds\\
&\quad+\varepsilon^2\int_0^{|\partial\Omega|}\int_0^1\left(\frac{KM}{2|\partial\Omega|^2}-\frac{|\Omega|}{|\partial\Omega|}+\varepsilon\tilde R(\varepsilon)\right)w_j(s,\xi)(\varphi\circ\psi_{\varepsilon}(s,\xi))(1-\varepsilon\xi\kappa(s))d\xi ds.
\end{split}
\end{equation*}
Hence, by \eqref{w0} we write
\begin{equation}\label{piece7}
\begin{split}
&\varepsilon\mu_j\int_{\Omega}\rho_{\varepsilon}v_{j,\varepsilon}\varphi dx\\
&\qquad=-\varepsilon\mu_j\int_0^{|\partial\Omega|}\int_0^1\frac{M^2}{2|\partial\Omega|^2}(\xi-1)^2(u_j\circ\psi_{\varepsilon}(s,0))(\varphi\circ\psi_{\varepsilon}(s,\xi))d\xi ds+R_5(\varepsilon),
\end{split}
\end{equation}
where
\begin{equation*}
\begin{split}
&R_5(\varepsilon):=\varepsilon^2\mu_j\int_{\omega_{\varepsilon}}v_{j,\varepsilon}\varphi dx
-\varepsilon^2\mu_j \int_0^{|\partial\Omega|}\int_0^1\frac{M}{|\partial\Omega|}\kappa(s)\xi w_j(s,\xi)(\varphi\circ\psi_{\varepsilon}(s,\xi))d\xi ds\\
&+\varepsilon^2\int_0^{|\partial\Omega|}\int_0^1\left(\frac{KM}{2|\partial\Omega|^2}-\frac{|\Omega|}{|\partial\Omega|}+\varepsilon\tilde R(\varepsilon)\right)w_j(s,\xi)(\varphi\circ\psi_{\varepsilon}(s,\xi))(1-\varepsilon\xi\kappa(s))d\xi ds.
\end{split}
\end{equation*}

We also find convenient to set
\begin{equation}\label{rem5}
R_6(\varepsilon):=\varepsilon^2\int_{\Omega}\rho_{\varepsilon}(\mu_j v_{j,\varepsilon}^1+\mu_j^1v_{j,\varepsilon}+\varepsilon\mu_j^1 v_{j,\varepsilon}^1)\varphi dx.
\end{equation}

Now, by \eqref{grad2} and \eqref{w0}, by the theorem on change of variables in integrals, and  by integrating by parts with respect to the variable $\xi$ we have that
\begin{equation*}
\begin{split}
&\varepsilon\int_{\Omega}\nabla v_{j,\varepsilon}\cdot\nabla\varphi dx=\varepsilon\int_{\omega_{\varepsilon}}\nabla v_{j,\varepsilon}\cdot\nabla\varphi dx\\
&\quad=\varepsilon^2\int_0^{|\partial\Omega|}\int_0^1\Big(\frac{1}{\varepsilon^2}\partial_{\xi}w_j(s,\xi)\partial_{\xi}(\varphi\circ\psi_{\varepsilon})(s,\xi)\\
&\qquad+\frac{\partial_s w_j(s,\xi)\partial_s(\varphi\circ\psi_{\varepsilon})(s,\xi)}{(1-\varepsilon\xi\kappa(s))^2}\Big)(1-\varepsilon\xi\kappa(s))d\xi ds\\
&\quad=\int_0^{|\partial\Omega|}\int_0^1\frac{M\mu_j}{|\partial\Omega|}(u_j\circ\psi_{\varepsilon}(s,0))(\varphi\circ\psi_{\varepsilon}(s,\xi)) d\xi ds-\int_{\partial\Omega}\frac{M\mu_j}{|\partial\Omega|}u_j\varphi d\sigma\\
&\qquad-\varepsilon\int_0^{|\partial\Omega|}\int_0^1\frac{\mu_j M}{|\partial\Omega|}\kappa(s)\left(2\xi-1\right)(u_j\circ\psi_{\varepsilon}(s,0))(\varphi\circ\psi_{\varepsilon}(s,\xi))d\xi ds\\
&\qquad+\varepsilon^2\int_0^{|\partial\Omega|}\int_0^1 \left(\frac{\partial_s w_j(s,\xi)\partial_s(\varphi\circ\psi_{\varepsilon})(s,\xi)}{(1-\varepsilon\xi\kappa(s))^2}\right)(1-\varepsilon\xi\kappa(s))d\xi ds.
\end{split}
\end{equation*}
We write
\begin{equation}\label{piece8}
\begin{split}
&\varepsilon\int_{\Omega}\nabla v_{j,\varepsilon}\cdot\nabla\varphi dx\\
&\quad=\int_0^{|\partial\Omega|}\int_0^1\frac{M\mu_j}{|\partial\Omega|}(u_j\circ\psi_{\varepsilon}(s,0))(\varphi\circ\psi_{\varepsilon}(s,\xi)) d\xi ds-\int_{\partial\Omega}\frac{M\mu_j}{|\partial\Omega|}u_j\varphi d\sigma\\
&\qquad-\varepsilon\int_0^{|\partial\Omega|}\int_0^1\frac{\mu_j M}{|\partial\Omega|}\kappa(s)\left(2\xi-1\right)(u_j\circ\psi_{\varepsilon}(s,0))(\varphi\circ\psi_{\varepsilon}(s,\xi))d\xi ds+R_7(\varepsilon),
\end{split}
\end{equation}
where
\begin{equation*}
R_7(\varepsilon):=\varepsilon^2\int_0^{|\partial\Omega|}\int_0^1 \left(\frac{\partial_s w_j(s,\xi)\partial_s(\varphi\circ\psi_{\varepsilon})(s,\xi)}{(1-\varepsilon\xi\kappa(s))^2}\right)(1-\varepsilon\xi\kappa(s))d\xi ds.
\end{equation*}
Analogously, from \eqref{grad2}, \eqref{w1}, by a change of variables in the integrals and integrating by parts with respect to the variable $\xi$, we see that
\begin{equation}\label{new9}
\begin{split}
&\varepsilon^2\int_{\Omega}\nabla v_{j,\varepsilon}^1\cdot\nabla\varphi dx=\varepsilon^2\int_{\omega_{\varepsilon}}\nabla v_{j,\varepsilon}^1\cdot\nabla\varphi dx\\
&=-\varepsilon\int_{\partial\Omega}\left(\frac{M\mu_j(K-|\partial\Omega|\kappa)}{2|\partial\Omega|^2}-\frac{2M^2\mu_j^2}{3|\partial\Omega|^2}+\frac{M\mu_j^1-|\Omega|\mu_j}{|\partial\Omega|}\right)u_j\varphi d\sigma\\
&\quad-\varepsilon\int_{\partial\Omega}\frac{M\mu_j}{|\partial\Omega|}u_j^1\varphi d\sigma\\
&\quad+\varepsilon\int_0^{|\partial\Omega|}\int_0^1\frac{M\mu_j\kappa(s)}{|\partial\Omega|}(\xi-1)(u_j\circ\psi_{\varepsilon}(s,0))(\varphi\circ\psi_{\varepsilon}(s,\xi)) d\xi ds\\
&\quad+\varepsilon\int_0^{|\partial\Omega|}\int_0^1\frac{M}{|\partial\Omega|}\Big(\mu_j (u_j^1\circ\psi_{\varepsilon}(s,0))-\frac{M\mu_j^2}{2|\partial\Omega|}(\xi-1)^2 (u_j\circ\psi_{\varepsilon}(s,0))\\
&\quad\qquad\qquad\qquad+\mu_j^1 (u_j\circ\psi_{\varepsilon}(s,0))-\frac{M\mu_j^2}{|\partial\Omega|}\xi (u_j\circ\psi_{\varepsilon}(s,0))-\frac{|\Omega|\mu_j}{M}(u_j\circ\psi_{\varepsilon}(s,0))\\
&\quad\qquad\qquad\qquad+\frac{K\mu_j}{2|\partial\Omega|}(u_j\circ\psi_{\varepsilon}(s,0))\Big)(\varphi\circ\psi_{\varepsilon}(s,\xi)) d\xi ds+R_8(\varepsilon),
\end{split}
\end{equation}
where
\begin{equation*}
\begin{split}
R_8(\varepsilon)&:=\varepsilon^3\int_0^{|\partial\Omega|}\int_0^1 \left(\frac{\partial_s w_j^1(s,\xi)\partial_s(\varphi\circ\psi_{\varepsilon})(s,\xi)}{(1-\varepsilon\xi\kappa(s))^2}\right)(1-\varepsilon\xi\kappa(s))d\xi ds\\
&\quad-\varepsilon^2 \int_0^{|\partial\Omega|}\int_0^1 \xi\kappa(s)\partial_{\xi}w_j^1(s,\xi)\partial_{\xi}(\varphi\circ\psi_{\varepsilon})(s,\xi)d\xi ds.
\end{split}
\end{equation*}

From \eqref{piece1}, \eqref{piece2}, \eqref{piece3}, \eqref{rem4}, \eqref{piece5}, \eqref{piece6}, \eqref{piece7}, \eqref{rem5}, \eqref{piece8} and \eqref{new9}, and by a standard computation, it follows that the right-hand side of the equality in \eqref{eprod} equals
\begin{equation*}
\frac{1}{|1+\mu_j+\varepsilon\mu_j^1|}\left|I_{1,\varepsilon}+I_{2,\varepsilon}+I_{3,\varepsilon}+I_{4,\varepsilon}+I_{5,\varepsilon}+I_{6,\varepsilon}+I_{7,\varepsilon}\right|
\end{equation*}
with
\footnotesize
\begin{eqnarray}
I_{1,\varepsilon}&=&\frac{M\mu_j}{|\partial\Omega|}\int_0^{|\partial\Omega|}\int_0^1\Big((u_j\circ\psi_{\varepsilon}(s,\xi))-(u_j\circ\psi_{\varepsilon}(s,0))\nonumber\\
&&\ \ \ \ \ \ \ \ \ \ \ \ +\varepsilon\frac{M\mu_j}{|\partial\Omega|}\xi (u_j\circ\psi_{\varepsilon}(s,0))\Big)(\varphi\circ\psi_{\varepsilon}(s,\xi))d\xi ds\label{f1}\\
I_{2,\varepsilon}&=&-\varepsilon\frac{M\mu_j}{|\partial\Omega|}\int_0^{|\partial\Omega|}\int_0^1\left((u_j\circ\psi_{\varepsilon}(s,\xi))\right.\nonumber\\
&&\left.\ \ \ \ \ \ \ \ \ \ \ \ -(u_j\circ\psi_{\varepsilon}(s,0))\right)(\varphi\circ\psi_{\varepsilon}(s,\xi))\xi\kappa(s)d\xi ds\label{f2}\\
I_{3,\varepsilon}&=&\varepsilon\frac{\mu_j^1 M}{|\partial\Omega|}\int_0^{|\partial\Omega|}\int_0^1\left((u_j\circ\psi_{\varepsilon}(s,\xi))-(u_j\circ\psi_{\varepsilon}(s,0))\right)(\varphi\circ\psi_{\varepsilon}(s,\xi))d\xi ds\label{f3}\\
I_{4,\varepsilon}&=&\varepsilon\frac{M\mu_j}{|\partial\Omega|}\int_0^{|\partial\Omega|}\int_0^1\left((u_j^1\circ\psi_{\varepsilon}(s,\xi))-(u_j^1\circ\psi_{\varepsilon}(s,0))\right)(\varphi\circ\psi_{\varepsilon}(s,\xi))d\xi ds\label{f5}\\
I_{5,\varepsilon}&=&-\varepsilon\frac{\mu_j|\Omega|}{|\partial\Omega|}\int_0^{|\partial\Omega|}\int_0^1\left((u_j\circ\psi_{\varepsilon}(s,\xi))-(u_j\circ\psi_{\varepsilon}(s,0))\right)(\varphi\circ\psi_{\varepsilon}(s,\xi))d\xi ds\label{f6}\\
I_{6,\varepsilon}&=&\varepsilon\frac{\mu_j K M}{2|\partial\Omega|^2}\int_0^{|\partial\Omega|}\int_0^1\left((u_j\circ\psi_{\varepsilon}(s,\xi))-(u_j\circ\psi_{\varepsilon}(s,0))\right)(\varphi\circ\psi_{\varepsilon}(s,\xi))d\xi ds.\label{f7}\\
I_{7,\varepsilon}&=&\sum_{k=1}^8R_k(\varepsilon)\label{f8}.
\end{eqnarray}
\normalsize

To prove the validity of the lemma we will show that there exists $C>0$ such that
\begin{equation}\label{aim2}
|I_{k,\varepsilon}|\leq C\varepsilon^2\|\varphi\|_{\varepsilon}\ \ \ \forall\varepsilon\in(0,\varepsilon_{\Omega,j})\,,\varphi\in\mathcal H_{\varepsilon}(\Omega)
\end{equation}
for all $k\in\left\{1,...,7\right\}$. Through the rest of the proof we find convenient to denote by $C$ a positive constant which does not depend on $\varepsilon$ and $\varphi$ and which may be re-defined line by line.

Since $\Omega$ is assumed to be of class $C^{3}$, $u_j$ is a solution of \eqref{Steklov} and $u_j^1$ is a solution of \eqref{u1_problem_simple_true}, a classical elliptic regularity argument shows that $u_j,u_j^1\in C^{2}(\overline\Omega)$ (see e.g., \cite{agmon1}). Thus we conclude that the terms \eqref{f2}-\eqref{f7} can be bounded from above by $C\varepsilon^2\|\varphi\|_{\varepsilon}$ by the same argument used to study $J_{8,\varepsilon}$ in the proof of Lemma \ref{C1} (cf.~\eqref{J801}-\eqref{J8}). Hence \eqref{aim2} holds for $k\in\left\{2,...,6\right\}$.

Now we estimate $I_{1,\varepsilon}$ (cf.~\eqref{f1}). It is convenient to pass to the coordinates $(s,t)$ by the change of variables $x=\psi(s,t)$. From the regularity assumptions on $\Omega$ we have that $\psi$ is of class $C^2$ from $[0,|\partial\Omega|)\times(0,\varepsilon)$ to $\mathbb R^2$, for all $\varepsilon\in(0,\varepsilon_{\Omega,j})$. Thus $u_j\circ\psi$ is of class $C^2$ from $[0,|\partial\Omega|)\times(0,\varepsilon)$ to $\mathbb R$, for all $\varepsilon\in(0,\varepsilon_{\Omega,j})$.  Therefore, for each $(s,t)\in[0,\partial\Omega)\times(0,\varepsilon)$, there exist $t^*\in(0,t)$ and $t^{**}\in(0,t^*)$ such that
\begin{equation*}
(u_j\circ\psi)(s,t)-(u_j\circ\psi)(s,0)=t\partial_t(u_j\circ\psi)(s,t^*)
\end{equation*}
and
\begin{equation*}
\partial_t(u_j\circ\psi)(s,t^*)-\partial_t(u_j\circ\psi)(s,0)=t^*\partial^2_t(u_j\circ\psi)(s,t^{**}).
\end{equation*}
Moreover we note that $\partial_t u_j(\psi(s,0))=-\partial_{\nu} u_j(\gamma(s))=-\frac{\mu_j M}{|\partial\Omega|}u_j(\psi(s,0))$.
Then by \eqref{f1} we have
\[
\begin{split}
I_{1,\varepsilon}&=\frac{M\mu_j}{|\partial\Omega|}\int_0^{|\partial\Omega|}\frac{1}{\varepsilon}\int_0^{\varepsilon}t \bigl(\partial_t(u_j\circ\psi)(s,t^*)-\partial_t(u_j\circ\psi)(s,0)\bigr)(\varphi\circ\psi(s,t))dt ds\\
&=\frac{M\mu_j}{|\partial\Omega|}\int_0^{|\partial\Omega|}\frac{1}{\varepsilon}\int_0^{\varepsilon}t\,t^* \partial^2_t(u_j\circ\psi)(s,t^{**})(\varphi\circ\psi(s,t))dt ds.
\end{split}
\]
Then, by a computation based on the H\"older inequality, we verify that
\[
\begin{split}
\left|I_{1,\varepsilon}\right|&\leq\frac{M\mu_j}{|\partial\Omega|}\frac{1}{\varepsilon}\int_0^{|\partial\Omega|}\int_0^{\varepsilon}t^2\left|\partial^2_t (u_j\circ\psi)(s,t^{**})\right|\left|\varphi\circ\psi(s,t)\right|dt ds\\
&\leq\frac{M\mu_j }{|\partial\Omega|}\|u_j\|_{C^2(\overline\Omega)}\int_0^{|\partial\Omega|}\int_0^{\varepsilon}\frac{t^2}{\varepsilon^{\frac{1}{2}}}\frac{\left|\varphi\circ\psi(s,t)\right|}{\varepsilon^{\frac{1}{2}}}dt ds\\
&\leq\frac{M\mu_j }{|\partial\Omega|}\|u_j\|_{C^2(\overline\Omega)}\left(\int_0^{|\partial\Omega|}\int_0^{\varepsilon}\frac{t^4}{\varepsilon}dt ds\right)^{\frac{1}{2}}
\left(\int_0^{|\partial\Omega|}\int_0^{\varepsilon}\frac{\left(\varphi\circ\psi(s,t)\right)^2}{\varepsilon}dt ds\right)^{\frac{1}{2}},\\
&\leq\frac{M\mu_j C}{|\partial\Omega|\sqrt{5}}\varepsilon^2\|\varphi\|_{\varepsilon}
\end{split}
\]
where in the latter inequality we have use the argument of \eqref{J8}.
We conclude that \eqref{aim2} holds with $k=1$.

In order to complete the proof we have to estimate \eqref{f8}. Since $I_{7,\varepsilon}=\sum_{k=1}^8R_k(\varepsilon)$, we consider separately each $R_k(\varepsilon)$ and we start with $R_1(\varepsilon)$. By the H\"older's inequality we can prove the following estimates for the first term in the definition \eqref{rem1} of $R_1(\varepsilon)$, 
\begin{equation*}
\left|\mu_j \varepsilon\tilde{R}(\varepsilon)\int_{\omega_{\varepsilon}}u_j\varphi dx\right|\le\frac{\mu_j \varepsilon\tilde{R}(\varepsilon)}{\tilde\rho_{\varepsilon}}\int_{\omega_{\varepsilon}}\tilde\rho_{\varepsilon} \left|u_j\varphi \right|dx\leq \frac{\mu_j \varepsilon\tilde{R}(\varepsilon)}{\tilde\rho_{\varepsilon}}\|u_j\|_{\varepsilon}\|\varphi\|_{\varepsilon}\leq C\varepsilon^2\|\varphi\|_{\varepsilon}.
\end{equation*}
For the second summand in the right-hand side of \eqref{rem1}, a computation based on the rule of change of variables in integrals and on H\"older's inequality shows that
\begin{multline}\label{est_ch_var}
{\left|\mu_j\varepsilon^2\int_0^{|\partial\Omega|}\int_0^1\frac{K M}{2|\partial\Omega|^2}\xi\kappa(s)(u_j\circ\psi_{\varepsilon}(s,\xi))(\varphi\circ\psi_{\varepsilon}(s,\xi)) d\xi ds\right|}\\
{\leq \varepsilon^2\frac{C}{\tilde\rho_{\varepsilon}}\int_{\omega_{\varepsilon}}\frac{\tilde\rho_{\varepsilon}}{\varepsilon}\left|u_j\varphi \right| dx\leq C\varepsilon^2\|u_j\|_{\varepsilon}\|\varphi\|_{\varepsilon}.}
\end{multline}
Analogously, for the third summand in the right-hand side of \eqref{rem1} we have
\begin{multline*}
\left|\mu_j\varepsilon^2\int_0^{|\partial\Omega|}\int_0^1\frac{|\Omega|}{|\partial\Omega|}\xi\kappa(s)(u_j\circ\psi_{\varepsilon}(s,\xi))(\varphi\circ\psi_{\varepsilon}(s,\xi)) d\xi ds\right|\leq C\varepsilon^2\|u_j\|_{\varepsilon}\|\varphi\|_{\varepsilon}.
\end{multline*}
This proves that $\left|R_1(\varepsilon)\right|\leq C\varepsilon^2\|\varphi\|_{\varepsilon}$. Let us now consider $R_2(\varepsilon)$. By H\"older's inequality and by Proposition \ref{L2<eps}, one deduces the following inequality for the first term in the definition  \eqref{rem2} of $R_2(\varepsilon)$,
\begin{equation*}
\left|\varepsilon^2\mu_j\int_{\Omega} u_j^1\varphi dx\right|\leq C\varepsilon^2\|u_j^1\|_{L^2(\Omega)}\|\varphi\|_{L^2(\Omega)}\leq C\varepsilon^2\|\varphi\|_{\varepsilon}.
\end{equation*}
For the second summand in the right-hand side of \eqref{rem2} we observe that, by an argument based on the H\"older inequality, we have
\begin{multline*}
\left|\mu_j\varepsilon\int_{\omega_{\varepsilon}}\left(\frac{KM}{2|\partial\Omega|}-\frac{|\Omega|}{|\partial\Omega|}+\varepsilon\tilde R(\varepsilon)\right)u_j^1\varphi dx\right|\\
{\leq C\mu_j\frac{\varepsilon^2}{\tilde\rho_{\varepsilon}}\int_{\omega_{\varepsilon}}\frac{\tilde\rho_{\varepsilon}}{\varepsilon}\left|u^1_j\varphi \right| dx\leq C\varepsilon^2\|u^1_j\|_{\varepsilon}\|\varphi\|_{\varepsilon}\le C\varepsilon^2\|\varphi\|_{\varepsilon},}
\end{multline*}
where in the latter inequality we have used the fact that 
\begin{equation*}
\|u^1_j\|_{\varepsilon}\leq C\qquad\forall\varepsilon\in(0,\varepsilon_{\Omega,j})
\end{equation*}
for some $C>0$, a fact that can be proved by arguing as in \eqref{J1e_eq2}, \eqref{J1e_eq2.1}.

By a similar argument, we can also prove that the third summand in the right-hand side of \eqref{rem2} is smaller than $C\varepsilon^2\|\varphi\|_{\varepsilon}$. Hence, we deduce that
\[
\left|R_2(\varepsilon)\right|\leq C\varepsilon^2\|\varphi\|_{\varepsilon}.
\] 

The proof that  $|R_k(\varepsilon)|\leq C\varepsilon^2\|\varphi\|_{\varepsilon}$ for $k\in\left\{3,...,6\right\}$ can be effected by a straightforward modification of the prove that  $|R_k(\varepsilon)|\leq C\varepsilon^2\|\varphi\|_{\varepsilon}$ for $k\in\left\{1,2\right\}$ and by exploiting Lemmas \ref{vjeL2} and \ref{vj1eL2}.

We now  consider $R_7(\varepsilon)$. By integrating by parts with respect to the variable $s$, we have 
\[
\begin{split}
&R_7(\varepsilon)=\varepsilon^2\int_0^{|\partial\Omega|}\int_0^1\frac{\partial_s w_j(s,\xi)\partial_s(\varphi\circ\psi_{\varepsilon})(s,\xi)}{(1-\varepsilon\xi\kappa(s))} d\xi ds\\
&\ =-\varepsilon^2\int_0^{|\partial\Omega|}\int_0^1\left(\frac{(1-\varepsilon\xi\kappa(s))\partial^2_sw_j(s,\xi)+\varepsilon\xi\kappa'(s)\partial_sw_j(s,\xi)}{(1-\varepsilon\xi\kappa(s))^2}\right)(\varphi\circ\psi_{\varepsilon}(s,\xi))d\xi ds.
\end{split}
\]
Then, by \eqref{positiveinf} and since $\kappa$ and $\kappa'$ are bounded on $(0,|\partial\Omega|)$ (because $\Omega$ is of class $C^3$) we deduce that 
\[
\left|R_7(\varepsilon)\right|\leq C\varepsilon^2\int_0^{|\partial\Omega|}\int_0^1\left(\left|\partial^2_s w_j\right|+\left|\partial_sw_j\right|\right)\left|\varphi\right| d\xi ds.
\]
Hence, by the definition of $w_j$ in \eqref{w0} and by a computation based on the H\"older inequality (see also \eqref{est_ch_var}) we find that
\[
\left|R_7(\varepsilon)\right|\leq C\varepsilon^2\|u_j\|_{C^2(\overline\Omega)}\|\varphi\|_{\varepsilon}\,.
\] 
We conclude that $\left|R_7(\varepsilon)\right|\leq C\varepsilon^2\|\varphi\|_{\varepsilon}$. In a similar way one can show that $\left|R_8(\varepsilon)\right|\leq C\varepsilon^2\|\varphi\|_{\varepsilon}$. The proof of the lemma is now complete.

\endproof
\end{lemma}

In the next step we verify that $\|u_j+\varepsilon v_{j,\varepsilon}+\varepsilon u_j^1+\varepsilon^2 v_{j,\varepsilon}^1\|_{\varepsilon}^2-\frac{M}{|\partial\Omega|}\left(1+\mu_j+\varepsilon\mu_j^1\right)$ is in $O(\varepsilon^2)$ for $\varepsilon\rightarrow 0$. {}

\begin{lemma}\label{propedeutico_step2}

There exists a constant $C>0$ such that
\begin{equation*}
\left|\|u_j+\varepsilon v_{j,\varepsilon}+\varepsilon u_j^1+\varepsilon^2 v_{j,\varepsilon}^1\|_{\varepsilon}^2-\frac{M}{|\partial\Omega|}\left(1+\mu_j+\varepsilon\mu_j^1\right)\right|\leq C\varepsilon^2,\ \ \ \forall\varepsilon\in (0,\varepsilon_{\Omega,j}).
\end{equation*}
\proof
A straightforward computation shows that
$$
\|u_j+\varepsilon v_{j,\varepsilon}+\varepsilon u_j^1+\varepsilon^2 v_{j,\varepsilon}^1\|_{\varepsilon}^2=\sum_{k=1}^8N_{k,\varepsilon},
$$
where
\footnotesize

\[
\begin{split}
N_{1,\varepsilon}&:=\varepsilon\int_{\Omega}2\varepsilon u_ju_j^1+\varepsilon^2\left(u_j^1\right)^2dx+\int_{\Omega}\varepsilon^2|\nabla u_j^1|^2dx\\
&\quad+\varepsilon\int_{\omega_{\varepsilon}}\left(\varepsilon^2 v_{j,\varepsilon}^2+\varepsilon^4\left(v_{j,\varepsilon}^1\right)^2+2\varepsilon u_j v_{j,\varepsilon}+2\varepsilon^2 u_j^1 v_{j,\varepsilon}+2\varepsilon^3 u_j^1v_{j,\varepsilon}^1+2\varepsilon^2 u_j v_{j,\varepsilon}^1+2\varepsilon^3 v_{j,\varepsilon} v_{j,\varepsilon}^1\right)dx\\
&\quad+\int_{\omega_{\varepsilon}}\frac{\tilde\rho_{\varepsilon}}{\varepsilon}\left(\varepsilon^2\left(u_j^1\right)^2+\varepsilon^2 v_{j,\varepsilon}^2+\varepsilon^4\left(v_{j,\varepsilon}^1\right)^2+2\varepsilon^2 u_j^1 v_{j,\varepsilon}+2\varepsilon^3 u_j^1v_{j,\varepsilon}^1+2\varepsilon^2 u_j v_{j,\varepsilon}^1+2\varepsilon^3 v_{j,\varepsilon} v_{j,\varepsilon}^1\right)dx\\
&\quad+\int_{\omega_{\varepsilon}}\varepsilon^4|\nabla v_{j,\varepsilon}^1|^2+2\varepsilon^2\nabla u_j\cdot\nabla v_{j,\varepsilon}^1+2\varepsilon^2\nabla v_{j,\varepsilon}\nabla u_j^1+2\varepsilon^3\nabla v_{j,\varepsilon}\cdot\nabla v_{j,\varepsilon}^1+2\varepsilon^3\nabla u_j^1\cdot\nabla v_{j,\varepsilon}^1dx,\\
N_{2,\varepsilon}&:=\varepsilon\int_{\Omega}u_j^2dx,\\
N_{3,\varepsilon}&:=\int_{\omega_{\varepsilon}}\frac{\tilde\rho_{\varepsilon}}{\varepsilon}u_j^2 dx,\\
N_{4,\varepsilon}&:=2\int_{\omega_{\varepsilon}}\tilde\rho_{\varepsilon}u_j u_j^1 dx,\\
N_{5,\varepsilon}&:=2\int_{\omega_{\varepsilon}}\tilde\rho_{\varepsilon}u_j v_{j,\varepsilon} dx,\\
N_{6,\varepsilon}&:=\int_{\Omega}|\nabla u_j|^2+2\varepsilon\nabla u_j\cdot\nabla u_j^1\,dx,\\
N_{7,\varepsilon}&:=\varepsilon^2\int_{\omega_{\varepsilon}}|\nabla v_{j,\varepsilon}|^2dx,\\
N_{8,\varepsilon}&:=2\varepsilon\int_{\omega_{\varepsilon}}\nabla u_j\cdot\nabla v_{j,\varepsilon} dx.
\end{split}
\]\normalsize
We begin by considering $N_{1,\varepsilon}$. By standard elliptic regularity (see \cite{agmon1}) the functions $u_j$ and $u_j^1$ are of class $C^2$ on $\overline\Omega$. Then, by Propositions  \ref{vjeL2},  \ref{vjee}, \ref{vj1eL2},   and \ref{vj2ee}, and by a standard computation one shows that
\[
\left|N_{1,\varepsilon}\right|\le C\varepsilon^2
\]
for some $C>0$.

We now re-write the $N_{k,\varepsilon}$'s with  $k\in\left\{3,...,8\right\}$, in a more suitable way. 
We start with $N_{3,\varepsilon}$. By the  membership of $u_j$ in $C^{2}(\overline\Omega)$ and by the definition of the change of variable $\psi$, we deduce that the map from $[0,\varepsilon]$ to $\mathbb{R}$ which takes $t$ to $u_j\circ\psi(s,t)$ is of class $C^2$. Then, by the Taylor formula we have
\[
(u_j\circ\psi(s,{t}))^2=(u_j\circ\psi(s,0))^2+2tu_j(\psi(s,0))\partial_t(u_j\circ\psi)(s,0)+F(t)t^2\qquad\forall t\in[0,\varepsilon]\,,
\]
 where $F(t)\in C(\overline{[0,\varepsilon]})$. Since $u_j$ is a solution of \eqref{Steklov} it follows that
\begin{equation}\label{ujpsiTaylor}
(u_j\circ\psi(s,{t}))^2=(u_j\circ\psi(s,0))^2+2t\frac{M\mu_j}{|\partial\Omega|}(u_j\circ\psi(s,0))^2+F(t)t^2.
\end{equation}
Then, by the definition of $N_{3,\varepsilon}$, by \eqref{ujpsiTaylor}, and by the expansion \eqref{asymptotic_rho} of $\tilde\rho_{\varepsilon}$, we deduce that
\normalsize
\[
\begin{split}
N_{3,\varepsilon}&
=\int_0^{|\partial\Omega|}\int_0^{\varepsilon}\left(\frac{M}{\varepsilon|\partial\Omega|}+\frac{\frac{1}{2}KM-|\Omega||\partial\Omega|}{|\partial\Omega|^2}+{\varepsilon}\tilde{R}(\varepsilon)\right)\\
&\qquad\times \left((u_j\circ\psi(s,0))^2+2t\frac{M\mu_j}{|\partial\Omega|}(u_j\circ\psi(s,0))^2+C(t)t^2\right) (1-t\kappa(s))dtds
\end{split}
\]
\normalsize
From standard computations and recalling that $\int_{\partial\Omega}u_j^2 d\sigma=1$, it follows that 
$$
N_{3,\varepsilon}=\frac{M}{|\partial\Omega|}+\varepsilon\left(\frac{\frac{1}{2}KM-|\Omega||\partial\Omega|}{|\partial\Omega|^2}\right)-\varepsilon\frac{M^2\mu}{|\partial\Omega|^2}-\varepsilon\frac{M}{2|\partial\Omega|}\int_{\partial\Omega}u_j^2\kappa{\circ\gamma^{(-1)}} d\sigma+Q_{3,\varepsilon},
$$
where $Q_{3,\varepsilon}$ satisfies the inequality
$$
|Q_{3,\varepsilon}|\leq C\varepsilon^2
$$

By a similar computation and by exploiting \eqref{w0} and \eqref{uj1condition}, one can also verify that
$$
N_{4,\varepsilon}=\varepsilon \frac{2M}{|\partial\Omega|}\left(\frac{\mu_j^1}{2\mu_j}+\frac{M\mu_j}{3|\partial\Omega|}\right)+Q_{4,\varepsilon}
$$
with 
$$
|Q_{4,\varepsilon}|\leq C\varepsilon^2
$$
and that 
$$
N_{5,\varepsilon}=-\varepsilon\frac{M^2\mu_j}{3|\partial\Omega|^2}+Q_{5,\varepsilon},
$$
with
$$
|Q_{5,\varepsilon}|\leq C\varepsilon^2.
$$

Now we turn to consider $N_{6,\varepsilon}$. By a computation based on the divergence theorem and on equality $\partial_\nu u_j=\frac{M\mu_j}{|\partial\Omega|}u_j$, we find that
\[
\begin{split}
N_{6,\varepsilon}&=\frac{M\mu_j}{|\partial\Omega|}\int_{\partial\Omega}u_j^2d\sigma+2\varepsilon\frac{M\mu_j}{|\partial\Omega|}\int_{\partial\Omega}u_ju_j^1d\sigma+Q_{6,\varepsilon}\\
&=\frac{M\mu_j}{|\partial\Omega|}+\varepsilon\left(\frac{M\mu_j^1}{|\partial\Omega|}+\frac{2M^2\mu_j^2}{3|\partial\Omega|^2}\right)+Q_{6,\varepsilon},
\end{split}
\]
with
$$
Q_{6,\varepsilon}:=\varepsilon^2\int_{\Omega}|\nabla u_j^1|^2dx.
$$
Since  $u_j^1\in C^2(\overline\Omega)$, we deduce that
$$
|Q_{6,\varepsilon}|\leq C\varepsilon^2.
$$

Next we consider $N_{7,\varepsilon}$. By passing to coordinates $(s,\xi)$ in the definition of $N_{7,\varepsilon}$, and by using formulas \eqref{grad2} and \eqref{w0}, one shows that
$$
N_{7,\varepsilon}=\varepsilon\frac{M^2\mu_j^2}{3|\partial\Omega|^2}+Q_{7,\varepsilon}
$$
with
$$
|Q_{7,\varepsilon}|\leq C \varepsilon^2.
$$

Finally we consider $N_{8,\varepsilon}$. By the membership of $u_j$ in $C^2(\overline\Omega)$ and by equality $\partial_\nu u_j=\frac{M\mu_j}{|\partial\Omega|}u_j$ we have 
\[
\partial_\xi u_j\circ\psi_\varepsilon(s,\xi)=\varepsilon(\partial_\nu u_j)\circ\psi_\epsilon(s,0)+\epsilon^2\tilde U(s,\xi)=\varepsilon\frac{M\mu_j}{|\partial\Omega|}u_j\circ\psi_\epsilon(s,0)+\epsilon^2\tilde U(s,\xi),
\]
where $\tilde U$ is a continuous function on $[0,|\partial\Omega|]\times[0,1]$. Then, by passing to coordinates $(s,\xi)$ in the definition of $N_{8,\varepsilon}$, and by using formulas \eqref{grad2} and \eqref{w0}, one verifies that
$$
N_{8,\varepsilon}=-\varepsilon\frac{M^2\mu_j^2}{|\partial\Omega|^2}+Q_{8,\varepsilon},
$$
with
$$
|Q_{8,\varepsilon}|\leq C \varepsilon^2.
$$

Now we set $Q_\varepsilon:=N_{1,\varepsilon}+\sum_{k=3}^8 Q_{k,\varepsilon}$. A straightforward computation shows that
\begin{equation*}
\begin{split}
&\|u_j+\varepsilon v_{j,\varepsilon}+\varepsilon u_j^1+\varepsilon^2 v_{j,\varepsilon}^1\|_{\varepsilon}^2-\frac{M}{|\partial\Omega|}\left(1+\mu_j+\varepsilon\mu_j^1\right)\\
&\quad=\frac{M}{|\partial\Omega|}\Bigg[1+\mu_j+\varepsilon\Bigg(\frac{|\partial\Omega|}{M}\int_{\Omega}u_j^2dx\\
&\qquad-\frac{|\Omega|}{M}-\frac{2M\mu_j}{3|\partial\Omega|}+\frac{K}{2|\partial\Omega|}-\frac{1}{2}\int_{\partial\Omega}u_j^2\kappa{\circ\gamma^{(-1)}} d\sigma+\frac{\mu_j^1}{\mu_j}+\mu_j^1\Bigg)\Bigg]\\
&\qquad-\frac{M}{|\partial\Omega|}\left(1+\mu_j+\varepsilon\mu_j^1\right)+Q_\varepsilon.
\end{split}
\end{equation*}
We note that by \eqref{top_der_formula} we have
$$
\frac{\mu_j^1}{\mu_j}=-\frac{|\partial\Omega|}{M}\int_{\Omega}u_j^2dx+\frac{|\Omega|}{M}-\frac{2M\mu_j}{3|\partial\Omega|}-\frac{K}{2|\partial\Omega|}+\frac{1}{2}\int_{\partial\Omega}u_j^2\kappa\circ\gamma^{(-1)} d\sigma,
$$
therefore
$$
\|u_j+\varepsilon v_{j,\varepsilon}+\varepsilon u_j^1+\varepsilon^2 v_{j,\varepsilon}^1\|_{\varepsilon}^2-\frac{M}{|\partial\Omega|}\left(1+\mu_j+\varepsilon\mu_j^1\right)=Q_\varepsilon.
$$
The conclusion of the proof of the lemma follows by observing that $|Q_\varepsilon|\leq C\varepsilon^2$ for all $\varepsilon\in(0,\varepsilon_{\Omega,j})$.
\endproof
\end{lemma}

We are now ready to prove Theorems \ref{asymptotic_eigenvalues} and \ref{asymptotic_eigenfunctions} by Lemma \ref{lemma_fondamentale}.

\noindent{\em Proof of Theorems \ref{asymptotic_eigenvalues} and \ref{asymptotic_eigenfunctions}.}
We first prove \eqref{expansion_eigenvalues}. By a standard continuity argument it follows that there exists $\varepsilon_{\mu_j^1}\in(0,\varepsilon_{\Omega,j})$ such that
$$
1+\mu_j+\varepsilon\mu_j^1>1
$$
for all $\varepsilon\in(0,\varepsilon_{\mu_j^1})$. By Lemma \ref{propedeutico_step2} there exists $\varepsilon_j^*\in(0,\varepsilon_{\mu_j^1})$ such that
$$
\|u_j+\varepsilon v_{j,\varepsilon}+\varepsilon u_j^1+\varepsilon^2 v_{j,\varepsilon}^1\|_{\varepsilon}>\frac{1}{2}\sqrt{\frac{M}{|\partial\Omega|}}\left(1+\mu_j+\varepsilon\mu_j^1\right)^{\frac{1}{2}}\ \ \ \forall\varepsilon\in(0,\varepsilon_j^*).
$$
By multiplying  both sides of \eqref{condition_2} by $\|u_j+\varepsilon v_{j,\varepsilon}+\varepsilon u_j^1+\varepsilon^2 v_{j,\varepsilon}^1\|_{\varepsilon}^{-1}$ we deduce that

\begin{equation}\label{plug_step2}
\begin{split}
&\left|\left\langle\mathcal A_{\varepsilon}\left(\frac{u_j+\varepsilon v_{j,\varepsilon}+\varepsilon u_j^1+\varepsilon^2 v_{j,\varepsilon}^1}{\|u_j+\varepsilon v_{j,\varepsilon}+\varepsilon u_j^1+\varepsilon^2 v_{j,\varepsilon}^1\|_{\varepsilon}}\right)\right.\right.\\
&\qquad\qquad\left.\left.-\frac{1}{1+\mu_j+\varepsilon\mu_j^1}\left(\frac{u_j+\varepsilon v_{j,\varepsilon}+\varepsilon u_j^1+\varepsilon^2 v_{j,\varepsilon}^1}{\|u_j+\varepsilon v_{j,\varepsilon}+\varepsilon u_j^1+\varepsilon^2 v_{j,\varepsilon}^1\|_{\varepsilon}}\right),\varphi\right\rangle_{\varepsilon}\right|\leq C_7\varepsilon^2 \|\varphi\|_{\varepsilon},
\end{split}
\end{equation}
for all $\varphi\in H^1(\Omega)$ and $\varepsilon\in(0,\varepsilon_j^*)$ with $C_7:=2\sqrt{\frac{|\partial\Omega|}{M}}(1+\mu_j+\varepsilon\mu_j^1)^{-\frac{1}{2}}C_6$. By taking 
\[
\varphi=\mathcal A_{\varepsilon}\left(\frac{u_j+\varepsilon v_{j,\varepsilon}+\varepsilon u_j^1+\varepsilon^2 v_{j,\varepsilon}^1}{\|u_j+\varepsilon v_{j,\varepsilon}+\varepsilon u_j^1+\varepsilon^2 v_{j,\varepsilon}^1\|_{\varepsilon}}\right)-\frac{1}{1+\mu_j+\varepsilon\mu_j^1}\left(\frac{u_j+\varepsilon v_{j,\varepsilon}+\varepsilon u_j^1+\varepsilon^2 v_{j,\varepsilon}^1}{\|u_j+\varepsilon v_{j,\varepsilon}+\varepsilon u_j^1+\varepsilon^2 v_{j,\varepsilon}^1\|_{\varepsilon}}\right)
\] in \eqref{plug_step2}, we obtain
\begin{multline*}
\left\|\mathcal A_{\varepsilon}\left(\frac{u_j+\varepsilon v_{j,\varepsilon}+\varepsilon u_j^1+\varepsilon^2 v_{j,\varepsilon}^1}{\|u_j+\varepsilon v_{j,\varepsilon}+\varepsilon u_j^1+\varepsilon^2 v_{j,\varepsilon}^1\|_{\varepsilon}}\right)\right.\\
\left.-\frac{1}{1+\mu_j+\varepsilon\mu_j^1}\left(\frac{u_j+\varepsilon v_{j,\varepsilon}+\varepsilon u_j^1+\varepsilon^2 v_{j,\varepsilon}^1}{\|u_j+\varepsilon v_{j,\varepsilon}+\varepsilon u_j^1+\varepsilon^2 v_{j,\varepsilon}^1\|_{\varepsilon}}\right)\right\|_{\varepsilon}
\leq C_7\varepsilon^2.
\end{multline*}
As a consequence, we see that the assumptions of Lemma \ref{lemma_fondamentale} hold with $A=\mathcal A_{\varepsilon}$, $H=\mathcal H_{\varepsilon}(\Omega)$, $\eta=\frac{1}{1+\mu_j+\varepsilon\mu_j^1}$, $u=\frac{u_j+\varepsilon v_{j,\varepsilon}+\varepsilon u_j^1+\varepsilon^2 v_{j,\varepsilon}^1}{\|u_j+\varepsilon v_{j,\varepsilon}+\varepsilon u_j^1+\varepsilon^2 v_{j,\varepsilon}^1\|_{\varepsilon}}$, $r=C_7\varepsilon^2$ with $\varepsilon\in(0,\varepsilon_j^*)$. Accordingly, for all $\varepsilon\in(0,\varepsilon_j^*)$ there exists an eigenvalue $\eta^*$ of $\mathcal A_{\varepsilon}$ such that
\begin{equation}\label{quasi_auto_2}
\left|\frac{1}{1+\mu_j+\varepsilon\mu_j^1}-\eta^*\right|\leq C_7\varepsilon^2.
\end{equation}
Now we take $\varepsilon^{\sharp}_{\Omega,j}:=\min\left\{\varepsilon_j^*,\delta_j,C_7^{-1}r_j^*\right\}$ with $\delta_j$ and $r_j^*$ as in Lemma \ref{only2}. By \eqref{quasi_auto_2} and Lemma \ref{only2}, the eigenvalue $\eta^*_{\varepsilon}$ has to coincide with $\frac{1}{1+\lambda_j(\varepsilon)}$ for all $\varepsilon\in(0,\varepsilon^{\sharp}_{\Omega,j})$. It follows that
\begin{equation*}
|\lambda_j(\varepsilon)-\mu_j-\varepsilon\mu_j^1|\leq C_7|(1+\mu_j+\mu_j^1\varepsilon)(1+\lambda_j(\varepsilon))|\varepsilon^2\ \ \ \forall\varepsilon\in (0,\varepsilon^{\sharp}_{\Omega,j}).
\end{equation*}
The validity of \eqref{expansion_eigenvalues} follows from Theorem \ref{convergence} and by a straightforward computation.

We now consider \eqref{expansion_eigenfunctions}. By Lemma \ref{lemma_fondamentale} with $r=r_j^*$ it follows that for all $\varepsilon\in(0,\varepsilon^{\sharp}_{\Omega,j})$, there exists a function $u_{\varepsilon}^*\in\mathcal H_{\varepsilon}(\Omega)$ with $\|u_{\varepsilon}^*\|_{\varepsilon}=1$ which belongs to the space generated by all the eigenfunctions of $\mathcal A_{\varepsilon}$ associated with eigenvalues contained in the segment $\left[\frac{1}{1+\mu_j+\varepsilon\mu_j^1}-r_j^*,\frac{1}{1+\mu_j+\varepsilon\mu_j^1}+r_j^*\right]$ and such that
\begin{equation}\label{unieq2}
\left\|u_{\varepsilon^*}-\frac{u_j+\varepsilon v_{j,\varepsilon}+\varepsilon u_j^1+\varepsilon^2 v_{j,\varepsilon}^1}{\|u_j+\varepsilon v_{j,\varepsilon}+\varepsilon u_j^1+\varepsilon^2 v_{j,\varepsilon}^1\|_{\varepsilon}}\right\|_{\varepsilon}\leq\frac{2C_7}{r_j^*}\varepsilon^2.
\end{equation}
Since $\varepsilon\in(0,\varepsilon_{\Omega,j}^{\sharp})$, Lemma \ref{only2} implies that $\frac{1}{1+\lambda_j(\varepsilon)}$ is the only eigenvalue of $\mathcal A_{\varepsilon}$ which belongs to the segment $\left[\frac{1}{1+\mu_j+\varepsilon\mu_j^1}-r_j^*,\frac{1}{1+\mu_j+\varepsilon\mu_j^1}+r_j^*\right]$. In addition $\lambda_j(\varepsilon)$ is simple for $\varepsilon<\varepsilon^{\sharp}_{\Omega,j}$ (because $\varepsilon^{\sharp}_{\Omega,j}\leq\varepsilon_{\Omega,j}$). It follows that $u_{\varepsilon}^*$ coincides with the only eigenfunction with norm one corresponding to $\lambda_j(\varepsilon)$, namely $u_{\varepsilon}^*=\frac{u_{j,\varepsilon}}{\|u_{j,\varepsilon}\|_{\varepsilon}}$. Thus by \eqref{unieq2}

\begin{equation}\label{r*step2}
\left\|\frac{u_{j,\varepsilon}}{\|u_{j,\varepsilon}\|_{\varepsilon}}-\frac{u_j+\varepsilon v_{j,\varepsilon}+\varepsilon u_j^1+\varepsilon^2 v_{j,\varepsilon}^1}{\|u_j+\varepsilon v_{j,\varepsilon}+\varepsilon u_j^1+\varepsilon^2 v_{j,\varepsilon}^1\|_{\varepsilon}}\right\|_{\varepsilon}\leq \frac{2C_7}{r_j^*}\varepsilon^2\ \ \ \forall\varepsilon\in(0,\varepsilon_j^{\sharp}).
\end{equation}
By exploiting \eqref{expansion_eigenvalues} and \eqref{r*step2} and by arguing as in the proof of \eqref{ineq_eigenfunctions_step1} (cf.~\eqref{intermediate.eq4}-\eqref{intermediate.eq6}), one can prove that
\begin{equation*}
\|u_{j,\varepsilon}-u_j-\varepsilon v_{j,\varepsilon}-\varepsilon u_j^1-\varepsilon^2v_{j,\varepsilon}^1\|_{L^2(\Omega)}\leq C_8\varepsilon^2,
\end{equation*}
for some $C_8>0$. Then the validity of \eqref{expansion_eigenfunctions} follows by Proposition \ref{vj1eL2}. This concludes the proof of Theorems \ref{asymptotic_eigenvalues} and \ref{asymptotic_eigenfunctions}.

\qed


\begin{appendices}
\section{}
Let $u_j$ be the unique eigenfunction associated with a simple eigenvalue $\mu_j$ of problem \eqref{Steklov} such that $\int_{\partial\Omega}u_j^2d\sigma=1$. We consider the following problem
\begin{equation}\label{u1_problem_simple}
\left\{\begin{array}{ll}
-\Delta u=f, & {\rm in}\ \Omega,\\
\partial_{\nu}u-\frac{M\mu_j}{|\partial\Omega|}u=g_1+\lambda g_2, & {\rm on}\ \partial\Omega,
\end{array}\right.
\end{equation}
where $f\in L^2(\Omega)$, $g_1,g_2\in L^2(\partial\Omega)$ are given data which satisfy the condition $\int_{\partial\Omega}g_2u_jd\sigma\ne0$, and where the unknowns are the scalar $\lambda$ and the function $u$. The weak formulation of problem \eqref{u1_problem_simple} reads: find $(\lambda,u)\in\mathbb R\times H^1(\Omega)$ such that
\begin{equation}\label{u1_weak}
\int_{\Omega}\nabla u\cdot\nabla\varphi dx-\frac{M\mu_j}{|\partial\Omega|}\int_{\partial\Omega}u\varphi d\sigma=\int_{\Omega}f\varphi dx+\int_{\partial\Omega}g_1\varphi d\sigma+\lambda\int_{\partial\Omega}g_2\varphi d\sigma,
\end{equation}
for all $\varphi\in H^1(\Omega)$. We have the following proposition.

\begin{proposition}\label{A1}
Problem \eqref{u1_problem_simple} admits a weak solution $(u,\lambda)\in H^1(\Omega)\times\mathbb R$ if and only if
\begin{equation}\label{lambda_appendix}
\lambda=-\left(\int_{\Omega}fu_j dx+\int_{\partial\Omega}g_1u_j d\sigma\right)\left(\int_{\partial\Omega}g_2u_j d\sigma\right)^{-1}.
\end{equation}
Moreover, if  $u$ is a solution of \eqref{u1_problem_simple}, then any other solution of \eqref{u1_problem_simple} is given by $u +  \alpha\, u_j$ for some $\alpha\in\mathbb R$.
\proof
Let $\mathcal A_1$ be the operator from $H^1(\Omega)$ to $H^1(\Omega)'$ which takes $u\in H^1(\Omega)$ to the functional $\mathcal A_1[u]$ defined by
\begin{equation*}
\mathcal A_1[u][\varphi]:=\int_{\Omega}\nabla u\cdot\nabla\varphi dx+\int_{\partial\Omega}u\varphi d\sigma\,,\ \varphi\in H^1(\Omega).
\end{equation*} 
As is well-known,{} $\mathcal A_1$ is a homeomorphism from $H^1(\Omega)$ to $H^1(\Omega)'$. Then we consider the trace operator ${\rm Tr}$ from $H^1(\Omega)$ to $L^2(\partial\Omega)${}, and the operator $\mathcal J$ from $L^2(\partial\Omega)$ to $H^1(\Omega)'$ defined by
$$
\mathcal J[u][\varphi]:=\int_{\partial\Omega}u\,{\rm Tr}[\varphi] d\sigma\,,\ \forall \varphi\in H^1(\Omega).
$$
We define the operator $\mathcal A_2$ from $H^1(\Omega)$ to $H^1(\Omega)'$ as
\begin{equation*}
\mathcal A_2:=-\left(1+\frac{M\mu_j}{|\partial\Omega|}\right)\mathcal J\circ{\rm Tr}.
\end{equation*}
Since ${\rm Tr}$ is compact and $\mathcal J$ is bounded, $\mathcal A_2$  is also compact. It follows that the operator $\mathcal A:=\mathcal A_1+\mathcal A_2$ from $H^1(\Omega)$ to $H^1(\Omega)'$ is Fredholm of index zero, being the compact perturbation of an invertible operator.  Now we denote by $B(\lambda)$ the element of $H^1(\Omega)'$ defined by
$$
B(\lambda)[\varphi]:=\int_{\Omega}f\varphi dx+\int_{\partial\Omega}g_1\,{\rm Tr}[\varphi] d\sigma+\lambda\int_{\partial\Omega}g_2\,{\rm Tr}[\varphi] d\sigma\,,\ \varphi\in H^1(\Omega).
$$
Problem \eqref{u1_weak} is recast into: find $(\lambda,u)\in\mathbb R\times H^1(\Omega)$ such that
\begin{equation*}
\mathcal A[u]=B(\lambda).
\end{equation*}
The kernel of $\mathcal A$ is finite dimensional and it is the space of those $u^*$ such that
$$
\int_{\Omega}\nabla u^*\cdot\nabla\varphi dx-\frac{M\mu_j}{|\partial\Omega|}\int_{\partial\Omega}u^*\,{\rm Tr}[\varphi] d\sigma=0\ \ \forall\varphi\in H^1(\Omega).
$$
Since we have assumed that $\mu_j$ is a simple eigenvalue associated with the eigenfunction $u_j$, it follows that the kernel of $\mathcal A$ coincides with the one dimensional subspace of $H^1(\Omega)$ generated by $u_j$. Therefore, problem \eqref{u1_problem_simple} has solution if and only if $B(\lambda)$ satisfies the equality
$$
B(\lambda)[u_j]=\int_{\Omega}fu_j dx+\int_{\partial\Omega}g_1u_j d\sigma+\lambda\int_{\partial\Omega}g_2u_j d\sigma=0.
$$
Since we have also assumed that $\int_{\partial\Omega}g_2u_jd\sigma\ne0$, it follows that problem \eqref{u1_weak} has solution if and only if $\lambda$ is given by \eqref{lambda_appendix}. To prove the last statement of the theorem we observe that the solution $u$ of problem \eqref{u1_weak} is defined up to elements in the kernel of  $\mathcal A$, which is generated by $u_j$.
\endproof
\end{proposition}

\section{}
In this section we consider the case when $\Omega$ coincides with the unit ball $B$ of $\mathbb R^2$. In this specific case the eigenvalues of problem \eqref{Steklov} are given by
\begin{equation*}
\mu_{2j-1}=\mu_{2j}=\frac{2\pi j}{M}\,,\ j\in\mathbb N\setminus\lbrace0\rbrace,
\end{equation*} 
while $\mu_{0}=0$  and, due to the symmetry of the problem, all the positive eigenvalues have multiplicity two (see, e.g., Girouard and Polterovich \cite{girouardpolterovich}). To investigate the problem, it is convenient to use polar coordinates $(r,\theta)\in[0,+\infty)\times[0,2\pi)$ in $\mathbb R^2$ and to introduce the corresponding change of variables $x=\phi_s(r,\theta)=(r\cos(\theta),r\sin(\theta))$. The eigenfunctions associated with the eigenvalue $\mu_{2j-1}=\mu_{2j}$ are the two-dimensional harmonic polynomials $u_{j,1}, u_{j,2}$ of degree $j$, which can be written in polar coordinates as
\begin{eqnarray*}
u_{j,1}(r,\theta)&=&r^j\cos(j\theta),\\
u_{j,2}(r,\theta)&=&r^j\sin(j\theta).
\end{eqnarray*}
Problem \eqref{Neumann} for $\Omega=B$ has been  considered in Lamberti and Provenzano \cite{lambertiprovenzano2,lambertiprovenzano1}. In such works it has been proved that all the eigenvalues of problem \eqref{Neumann} on $B$ have multiplicity which is an integer multiple of two, except the first one which is equal to zero and has multiplicity one. Moreover, for a fixed $j\in\mathbb N\setminus\lbrace 0\rbrace$, there exists $\varepsilon_{j}>0$ such that $\lambda_j(\varepsilon)$ has multiplicity two for all $\varepsilon\in(0,\varepsilon_{j})$ (see also Theorem \ref{convergence}). The positive eigenvalues of \eqref{Neumann} on $B$ can be labelled with two indexes $k$ and $l$ and denoted by $\lambda_{2k-1,l}(\varepsilon)=\lambda_{2k,l}(\varepsilon)$, for $k,l\in\mathbb N\setminus\lbrace 0\rbrace$. The corresponding eigenfunctions, which we denote by $u_{0,l,\varepsilon}, u_{k,l,\varepsilon,1}$ and $u_{k,l,\varepsilon,2}$ can be written in the following form
\begin{eqnarray*}
u_{0,l,\varepsilon}&=&R_{0,l}(r),\\
u_{k,l,\varepsilon,1}&=&R_{k,l}(r)\cos(k\theta),\\
u_{k,l,\varepsilon,2}&=&R_{k,l}(r)\sin(k\theta),
\end{eqnarray*}
where $R_{k,l}(r)$ are suitable linear combinations of Bessel Functions of the first and second species and order $k${}. Moreover, it has been proved that $\lambda_{2k-1,1}(\varepsilon)\rightarrow\mu_{2k-1}$, $\lambda_{2k,1}(\varepsilon)\rightarrow\mu_{2k}$, $\lambda_{2k-1,l}(\varepsilon)\rightarrow+\infty$, $\lambda_{2k,l}(\varepsilon)\rightarrow+\infty$ for $l\geq 2$, $u_{k,1,\varepsilon,1}\rightarrow u_{k,1}$ and $u_{\varepsilon,k,1,2}\rightarrow u_{k,2}$ in the $L^2(\Omega)$ sense, as $\varepsilon\rightarrow 0$.

We note that, in principle, Theorem \ref{asymptotic_eigenvalues} could not be applied to this case since all the eigenvalues are multiple. Nevertheless, we have the following result concerning the derivative of the eigenvalues of \eqref{Neumann} at $\varepsilon=0$ when $\Omega=B$.

\begin{theorem}[Lamberti and Provenzano \cite{lambertiprovenzano2,lambertiprovenzano1}]\label{thmball}
For the eigenvalues of problem \eqref{Neumann} on the unit ball $B$ we have the following asymptotic expansion
\begin{equation}\label{asymptotic_ball}
\begin{split}
\lambda_{2j-1,1}(\varepsilon)&=\mu_{2j-1}+\left(\frac{2j\mu_{2j-1}}{3}+\frac{\mu_{2j-1}^2}{2(j+1)}\right)\varepsilon+O(\varepsilon^2)\\
&=\frac{2\pi j}{M}+\frac{2j^2 \pi}{M}\left(\frac{2}{3}+\frac{\pi}{M(1+j)}\right)\varepsilon+O(\varepsilon^2),
\end{split}
\end{equation}
as $\varepsilon\rightarrow 0$. The same formula holds if we substitute $\lambda_{2j-1,1}(\varepsilon)$ and $\mu_{2j-1}$ with $\lambda_{2j,1}(\varepsilon)$ and $\mu_{2j}$ respectively.
\end{theorem}
The proof of Theorem \ref{thmball} is strictly related to the fact that $\Omega$ is a ball and relies on the use of Bessel functions which allow to recast problem \eqref{Neumann} in the form of an equation $\mathcal F(\lambda,\varepsilon)=0$ in the unknowns $\lambda,\varepsilon\in\mathbb R$. The method used in \cite{lambertiprovenzano2} requires standard but lengthy computations, suitable Taylor's expansions and estimates on the corresponding remainders, as well as recursive formulas for the cross-products of Bessel functions and their derivatives.

We note that the first term in the asymptotic expansion of all the eigenvalues of \eqref{Neumann} on $B$ is positive, therefore locally, near the limiting problem \eqref{Steklov}, the eigenvalues are decreasing. Hence, we can say that the Steklov eigenvalues $\mu_{j}$ minimize the Neumann eigenvalues $\lambda_j(\varepsilon)$ for $\varepsilon$ small enough. We note that this does not prove global monotonicity of $\lambda_j(\varepsilon)$, which in fact does not hold  for any $j$; see Figures \ref{fig1} and \ref{fig2}. 

We now observe that, if we plug  $u_j=\pi^{-\frac{1}{2}}(r^j\cos(j\theta))\circ\phi_s^{(-1)}$ into formula \eqref{top_der_formula} and we recall that   the mean curvature $\kappa$ of $\partial B$ is constant end equals $1$, then we re-obtain  equality \eqref{asymptotic_ball}. So we can say that, in a sense, Theorem  \ref{asymptotic_eigenvalues}  continues to hold also in the case when $\Omega$ is a ball, despite of the fact that the eigenvalues are in such case multiple.  This is not surprising. In fact, we could have replaced through all the paper the space $H^1(\Omega)$ with the space $H^1_j(\Omega)$ of those functions $u$ in $H^1(\Omega)$ which are orthogonal to $(r^j\cos(j\theta))\circ\phi_s^{(-1)}$ with respect to the $H^1(\Omega)$ scalar product. In this way the eigenvalue $\mu_{2j-1}$ becomes simple and an argument based on Theorem \ref{asymptotic_eigenvalues} could be applied to study the asymptotic behavior.

We also remark that formula \eqref{asymptotic_ball} for the derivatives of the eigenvalues when $\Omega=B$ has been generalized to dimension $N>2$ in \cite{lambertiprovenzano2}. Again, the proof relies on the use of Bessel functions and explicit computations. 

The method used in the present paper is more general and allows to find a formula for the derivative of the eigenvalues $\lambda(\varepsilon)$ of problem \eqref{Neumann} for a quite wide class of domains in $\mathbb R^2$. A generalization of such formula for domains in $\mathbb R^N$ for $N>2$, the boundary of which can be globally parametrized with the unit sphere $S^{N-1}\subset \mathbb R^N$, will be part of a future work.

\section*{}
\begin{figure}[!ht]
 \centering
    \includegraphics[width=0.6\textwidth]{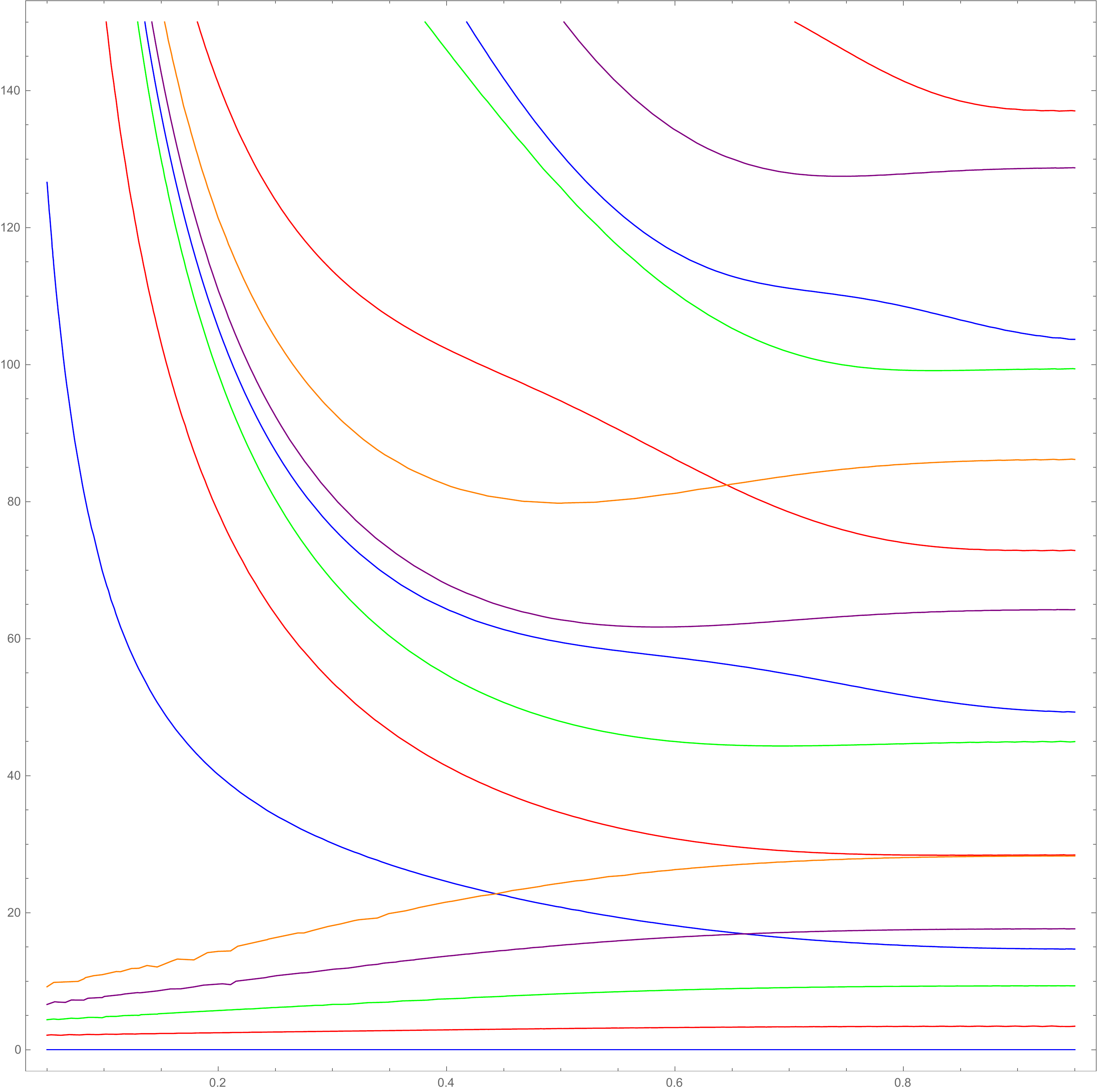}
		 \caption{{ {{ $\lambda_{2k-1,l}$}} with $M=\pi$ in the range $(\varepsilon,\lambda)\in(0,1)\times(0,150)$. In particular blue ($k=0,l=1,2,3,4$), red ($k=1,l=1,2,3,4$), green ($k=2,l=1,2,3$), purple ($k=3,l=1,2,3$), orange ($k=4,l=1,2$).} 
	}
		\label{fig1}
\end{figure}

\begin{figure}[!ht]
  \centering
   
      \includegraphics[width=0.6\textwidth]{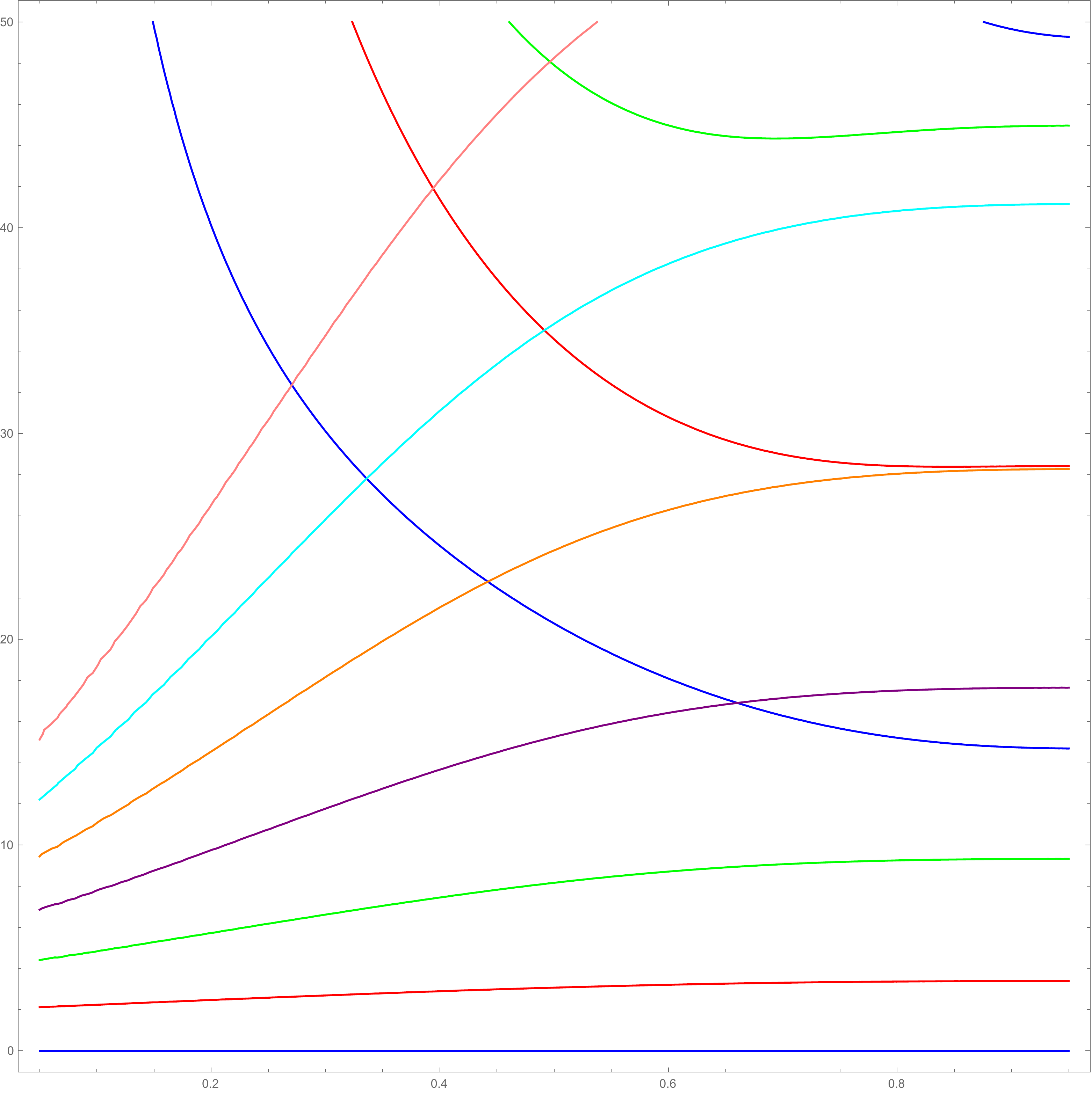}
		
  \caption{ {{$\lambda_{2k-1,l}$}} with $M=\pi$ in the range $(\varepsilon,\lambda)\in(0,1)\times(0,50)$. In particular blue ($k=0,l=1,2$), red ($k=1,l=1,2$), green ($k=2,l=1,2$), purple ($k=3,l=1$), orange ($k=4,l=1$), blue ($k=5,l=1$), pink ($k=6,l=1$) .}

					\label{fig2}
\end{figure}

\end{appendices}

\section*{Acknowledgements}	
The authors are deeply thankful to Professor Pier Domenico Lamberti and to Professor Sergei A.~Nazarov for the fruitful discussions on the topic. The authors also thank  the Center for Research and Development in Mathematics and Applications (CIDMA) of the University of Aveiro for the hospitality offered during the development of the work. In addition, the authors acknowledge the support of `Progetto di Ateneo: Singular perturbation problems for differential operators -- CPDA120171/12' - University of Padova. Matteo Dalla Riva acknowledges the support of HORIZON 2020 MSC EF project FAANon (grant agreement MSCA-IF-2014-EF-654795) at the University of Aberystwyth, UK. Luigi Provenzano acknowledges the financial support from the research project `INdAM GNAMPA Project 2015 - Un approccio funzionale analitico per problemi di perturbazione singolare e di omogeneizzazione'. Luigi Provenano is member of the Gruppo Nazionale
per l'Analisi Matematica, la Probabilit\`a e le loro Applicazioni (GNAMPA) of the Istituto Nazionale di Alta Matematica (INdAM).



\def\cprime{$'$} \def\cprime{$'$} \def\cprime{$'$} \def\cprime{$'$}
  \def\cprime{$'$}

\end{document}